\documentclass[12pt]{amsart}

\usepackage[T1]{fontenc}
\usepackage{amssymb}
\usepackage{bm}
\usepackage{ytableau}
\usepackage{amsmath}
\usepackage{verbatim}
 \usepackage{mathpple}
\usepackage{boondox-cal}
\usepackage{amsmath,amsfonts, amscd,amsthm, graphicx}
  \usepackage{amsrefs}
\usepackage[all]{xy}
\usepackage{color}
\usepackage{tikz-cd}

\usepackage{hyperref}

\def\XXint#1#2#3{{\setbox0=\hbox{$#1{#2#3}{\int}$}
\vcenter{\hbox{$#2#3$}}\kern-.5\wd0}}

\numberwithin{equation}{section}

\newcommand{\mbf}{\mathbf}
\newcommand{\mbb}{\mathbb}
\newcommand{\mf}{\mathfrak}
\newcommand{\mc}{\mathcal}

\theoremstyle{plain}
\newtheorem{thm}{Theorem}[section]

\newtheorem{thm-defn}{Theorem/Definition}[section]
\newtheorem{lem}[thm]{Lemma}
\newtheorem{lem-defn}[thm]{Lemma/Definition}
\newtheorem{prop}[thm]{Proposition}

\newtheorem{prop-defn}[thm]{Proposition-Definition}

\newtheorem{defn}[thm]{Definition}

\newtheorem{thm-alg}[thm]{Theorem/Algorithm}

\allowdisplaybreaks[4]  

\begin{document}

 \title{Rationality of the $K$-theoretical capped vertex function for Nakajima quiver varieties}
  \author{Tianqing Zhu}
  \address{Yau Mathematical Sciences Center}
\email{ztq20@mails.tsinghua.edu.cn}

  \date{}

  \maketitle

\begin{abstract} In this paper we prove the rationality of the capped vertex function with descendents for arbitrary Nakajima quiver varieties with generic stability conditions. We generalise the proof given by Smirnov to the general case, which requires to use techniques of tautological classes rather than the fixed-point basis. This result confirms that the "monodromy" of the capped vertex function is trivial, which gives a strong constraint for the monodromy of the capping operators. We will also provide a GIT wall-crossing formula for the capped vertex function in terms of the quantum difference operators and fusion operators.
\end{abstract}

\tableofcontents
\section{\textbf{Introduction}}
The $K$-theoretic enumerative geometry of the quasimap counting for GIT quotient has many ingredients that connects different subjects in mathematics and physics. In the theory of the quasimap counting, vertex function plays an important role in many aspects. It gives a strong connection between the enumerative geometry of the quiver varieties and the quantum integrable systems from the quantum affine algebras, for example in the finite type $A$, one can refer to 
\cite{KPSZ21} \cite{KZ18} \cite{PSZ20}.

The $K$-theoretic capped vertex functions are the generating series of relative quasimaps to Nakajima quiver varieties. It is defined in the non-localised $K$-theory, while the computation for the capped vertex function is much more difficult. 

It has been claimed and conjectured for a long time that the capped vertex function with descendents are Taylor expansions of some rational functions. The explicit formula was obtained for the case of Hilbert scheme of points on $\mbb{C}^2$ in \cite{AS24} and the equivariant Hilbert scheme of $[\mbb{C}^2/\mbb{Z}_l]$ in \cite{AD24}.

In this paper we prove that in general, the capped vertex function with descendents for the Nakajima quiver varieties $M_{\bm{\theta},0}(\mbf{v},\mbf{w})$ is a Taylor expansion of some rational functions over the Kahler variables $z$.

One of the application of the rationality of the capped vertex fucntion is that it implies that the monodromy of the capping operator is controlled by the elliptic stable envelopes. Since the rationality implies that the monodromy of the capped vertex function is trivial, and the monodromy of the vertex function is controlled by the elliptic stable envelope as stated in \cite{AO21}. The assumption that the capping operator has the monodromy controlled by the elliptic stable envelope has been used in lots of computations for the quantum difference equation, for example in the case of the Hilbert scheme of points $\text{Hilb}(\mbb{C}^2)$ one can refer to \cite{S24} and \cite{S24-1}.

\subsection{The capped vertex with descendents}
Let $\text{QM}^{\mbf{d}}_{\text{rel } p_2}(M_{\bm{\theta},0}(\mbf{v},\mbf{w}))$ be the moduli space of stable quasimaps from genus zero curve $C$ to Nakajima quiver varieties $M_{\bm{\theta},0}(\mbf{v},\mbf{w})$ relative at points $p_2\in C$ (For the definition see \ref{quasimap-moduli-space}).  It is equipped with a natural evaluation map $\hat{\text{ev}}_{p_i}:\text{QM}^{\mbf{d}}_{\text{rel }p_2}(M_{\bm{\theta},0}(\mbf{v},\mbf{w}))\rightarrow M_{\bm{\theta},0}(\mbf{v},\mbf{w})$ via evaluating the quasimap $f$ at $p_1$ or $p_2$.

The capped vertex function with a descendent $\tau$ is defined as:
\begin{align}
\text{Cap}^{\tau}(z):=\sum_{\mbf{d}}z^{\mbf{d}}\hat{\text{ev}}_{p_2,*}(\hat{\mc{O}}_{vir}^{\mbf{d}}\otimes\tau(\mc{V}|_{p_1}))\in K(X)[[z^{\mbf{d}}]]_{\mbf{d}\in C_{ample}}
\end{align}

It is a non-localised $K$-theory class since $\hat{\text{ev}}_p$ is a proper map. The variable $z=(z_1,\cdots,z_n)$ is usually called the Kahler variable.

By the definition of the quasimap moduli space and the capped vertex function:
\begin{align}
\text{Cap}^{(\tau)}(z)=\tau(\mc{V})\mc{K}_{X}^{1/2}+O(z)
\end{align}
where $\mc{K}_X$ is the canonical bundle for the quiver variety $X=M_{\bm{\theta},0}(\mbf{v},\mbf{w})$.

The capped vertex function is a simpler object than bare vertex function and capping operators in the $K$-theory. It has the large framing vanishing property which makes it classical when the framing vector $\mbf{w}$ is large:
\begin{thm}{(See Theorem 7.5.23 in \cite{O15})}
If the polarisation $T^{1/2}X$ is large with respect to $(\text{deg}(\tau)\bm{\theta}_{*})$, then
\begin{align}
\text{Cap}^{(\tau)}(z)=\tau(\mc{V})\mc{K}_X^{1/2}
\end{align}
\end{thm}

Here $\bm{\theta}_{*}\in C_{ample}\subset\text{Pic}(X)$ is an element in a given ample cone chosen in the Corollary 7.5.6 in \cite{O15}. The definition of "large" for the polarisation $T^{1/2}X$ is given in the Definition 7.5.15 in \cite{O15} or see section five.

The rationality of the capped vertex with descendents has been proved for the Jordan quiver varities in \cite{S16}. Moreover, the explicit formula for Jordan quiver variety and the affine type $A$ for the equivariant Hilbert scheme has been computed in \cite{AD24}\cite{AS24}. It has long been conjectured that the matrix coefficients of capped vertex with descendents is a rational function in $\mbb{Q}(z,a,t_e,q)_{e\in E}$.

In this paper, we prove the rationality conjecture of the capped vertex with descendents for the Nakajima quiver varieties:

\begin{thm}{(See Theorem \ref{Main-theorem})}
The matrix coefficients of the capped vertex function $\text{Cap}^{(\tau)}\in K_{T}(X)_{loc}[[z^{\mbf{d}}]]_{\mbf{d}\in\text{Pic}(X)}$ is a rational function in $K_{T}(X)_{loc}(z^{\mbf{d}})_{\mbf{d}\in\text{Pic}(X)}$ for arbitrary descendents $\tau\in K_{T}(pt)[\cdots,z_{i1}^{\pm1},\cdots,z_{in_i}^{\pm1},\cdots]^{\text{Sym}}$.
\end{thm}

As an important application, we can use the rationality of the capped vertex function to prove the rationality of the quantum difference operators $\mbf{M}_{\mc{L}}(z)$:

\begin{prop}{(See Proposition \ref{rationality-of-qde})}
The quantum difference operator $\mbf{M}_{\mc{L}}(z)$ has the matrix coefficients as the rational function of $K_{T}(X)_{loc}^{\otimes 2}(z^{\mbf{d}})_{\mbf{d}\in\text{Pic}(X)}$.
\end{prop}

Actually the rationality of the quantum difference operator is also important in the proof of the rationality of the capped vertex functions of arbitrary descedents. We will give a brief description later in the next subsection.

Another important application of the computation is that one can use the argument to give a GIT wall-crossing formula for the capped vertex function. Namely, given $\chi^{(\tau)}_{\bm{\theta}}(z)=\chi(\text{Cap}^{(\tau)}_{\bm{\theta}})(z)$, one can use the fusion operator $H^{(\mbf{w}_1),(\mbf{w}_2)}(z)$ defined in \ref{fusion-operator} to deduce the corresponding wall-crossing formula:

\begin{thm}{(See Theorem \ref{wall-crossing formula})}
For two stability condition $\bm{\theta}$ and $\bm{\theta}'$, the GIT wall-crossing for the Euler characteristic of the capped vertex function with descendents $\tau\in K_{T}(pt)[\cdots,x_{i1}^{\pm1},\cdots,x_{in_i}^{\pm1},\cdots]^{\text{Sym}}$ is given by:
\begin{equation}
\begin{aligned}
&\chi^{(\tau)}_{\bm{\theta}}(zp^{\mc{L}})-\chi^{(\tau)}_{\bm{\theta}'}(zp^{\mc{L}})\\
=&H_{\bm{\theta},\bm{\theta}'}^{(\tau)}+\frac{1}{|W|}\int_{\chi_{\bm{\theta}'}}\frac{\Delta_{\text{Weyl}}(s)}{\wedge^*(T^*\text{Rep}(\mbf{v},\mbf{w}))}q^{-\Omega}\times\\
&\times(\mbf{M}_{\mc{L}}^{\bm{\theta}}(zp^{-\mc{L}})^{-1}\mc{L}_{\bm{\theta}}H^{(\mbf{w}_1),(\mbf{w}_2)}_{\bm{\theta}}(zq^{-\frac{\Delta\mbf{w}_2}{2}})^{-1}\mc{L}^{-1}_{\bm{\theta}}-\mbf{M}_{\mc{L}}^{\bm{\theta}'}(zp^{-\mc{L}})^{-1}\mc{L}_{\bm{\theta}'}H^{(\mbf{w}_1),(\mbf{w}_2)}_{\bm{\theta}'}(zq^{-\frac{\Delta\mbf{w}_2}{2}})^{-1}\mc{L}^{-1}_{\bm{\theta}'})\times\\
&\times((\tau(\mc{V})\otimes\mc{L})\otimes1)\mc{K}^{1/2})d_{Haar}s
\end{aligned}
\end{equation}

Here $H^{(\tau)}_{\bm{\theta},\bm{\theta}'}(z)$ is the analytic continuation of $\chi^{(\tau)}_{\bm{\theta}}(z)$ via contour deformation from $\chi_{\bm{\theta}}$ to $\chi_{\bm{\theta}'}$. $\mc{L}$ is some ample line bundle such that it makes $\tau\otimes\mc{L}\in K_{T}(pt)[\cdots,x_{i1},\cdots,x_{in_i},\cdots]^{\text{Sym}}$.
\end{thm}

In many cases, $H^{(\mbf{w}_1),(\mbf{w}_2)}_{\bm{\theta}}(z)$ can have an explicit formula from the representation theory of the quantum affine algebras, i.e. \cite{S16} for $\theta=1$ of the Jordan quiver. While for the generic stability condition, it is still unknown whether the corresponding Maulik-Okounkov quantum affine algebra has the corresponding preprojective $K$-theoretic Hall algebra description or not. This theorem would turn the problem of the computation for the GIT wall-crossing to the construction of the algebraic description for the MO quantum affine algebra of different stability conditions. We will put this as the further study.

\subsection{Strategy of the proof}

The proof is based on the strategy given by Smirnov in \cite{S16}. There the main technique is to transform the problem to the quantum difference equations and wKZ equations. In the case of the instanton moduli space or Jordan quiver variety, the corresponding quantum difference equation can be constructed on the quantum toroidal algebra $U_{q,t}(\hat{\hat{\mf{gl}}}_1)$ as explained in \cite{Z24}.

The issue for the general case is the ambiguity for the wall $R$-matrix in the wKZ equation (See \ref{wKZ-equation}), in the case of the Jordan quiver, the corresponding $R$-matrix $R_0^-$ contains only one wall $R_{0}^-$, and this ensures the rationality of the fusion operator $H^{(r_1),(r_2)}(z)$. While in the general case, the matrix $R_{0}^{-}=\prod^{\leftarrow}_{w}R_{w}^{-}$ actually contains infinitely many walls, and to deal with the issue, we can show that when restricting to the component $K_{T}(M(\mbf{v},\mbf{w})^A)$, we can reduce the infinite product to the finite product, and thus it guarantees the rationality of the fusion operator $H^{(\mbf{w}_1),(\mbf{w}_2)}(z)$. For details see the proof of the Proposition \ref{rationality-for-wKZ}.

For the block diagonal part of the capping operator and the index limit of the vertex function, we use the standard techniques of the virtual localisation to prove this fact. This will give the rationality of the fusion operator.

Now since we have the rationality of the fusion operator $H^{(\mbf{w}_1),(\mbf{w}_2)}(z)$, combining the large framing vanishing for the capped vertex $\text{Cap}^{(\tau)}(z)$ and the factorisation property for the bare vertex and the factorisation for the capped vertex:
\begin{align}
\text{Cap}^{(\tau)}(a,z)=\Psi(a,z)\text{Vertex}^{(\tau)}(a,z)
\end{align}

One can deduce the main theorem for $\tau\in K_{T}(pt)[\cdots,x_{i1},\cdots,x_{in_i},\cdots]^{\text{Sym}}$. For the arbitrary $\tau$, we first prove the rationality of the quantum difference operator $\mbf{M}_{\mc{L}}(z)$, and use the translation symmetry of the vertex function and the quantum difference equation to deduce the theorem for arbitrary $\tau$.

For details of the proof see section \ref{section-6}.

\subsection{Outline of the paper}
The paper is organised as follows: In section two we introduce basic notions of the Nakajima quiver variety and the stable envelope over its equivariant $K$-theory. In section three we introduce some basic notions of the quasimap moduli space and the vertex function with descendents. We also prove the factorisation property for the vertex function.

In section four we introduce the nonabelian stable envelope over the Nakajima quiver variety. In section five we introduce the capped and capping operators, and also we introduce the fusion operators and the corresponding wKZ equation. We also introduce the "adjoint type" fusion operator as a digression.

In section six we prove the main theorem of the paper, and we prove the GIT wall-crossing formula for the capped vertex function.

\subsection{Acknowledgements}
The author would like to thank Yalong Cao, Andrei Smirnov, Hunter Dinkins, Yehao Zhou and Zijun Zhou for their helpful discussions. The author is supported by the international collaboration grant BMSTC and ACZSP(Grant no. Z221100002722017).

\section{\textbf{Quiver varieties and stable envelopes}}
\subsection{Definition}
Let $Q=(I,E)$ be a quiver with finite node set $I$. We allow the quiver $Q$ to have loops and multiple edges. We consider the following type of the double framed quiver:
\begin{align}
\overline{Q}^{fr}=(I\cup\overline{I}, E\cup I\cup\overline{E}\cup\overline{I})
\end{align} 
its corresponding representation space can be written as:
\begin{align}\label{quiver-representation}
\text{Rep}_{\overline{Q}^{fr}}(\mbf{v},\mbf{w})=T^*\text{Rep}_{Q^{fr}}(\mbf{v},\mbf{w})=\bigoplus_{ij\in E}\text{Hom}(V_i,V_j)\oplus\text{Hom}(V_j,V_i)\oplus\bigoplus_{i\in I}\text{Hom}(V_i,W_i)\oplus\text{Hom}(W_i,V_i)
\end{align}

This is a linear space with a natural symplectic form $\omega$. This symplectic form is preserved by the action of
\begin{align}
G_{\mbf{v}}=\prod_{i\in I}GL(V_i),\qquad G_{\mbf{w}}=\prod_{i\in I}GL(W_i)
\end{align}

It also defines an action of the group $G_{edge}:=\prod_{i\in I}Sp(2q_{ii})\times\prod_{ij\in E,i\neq j}GL(q_{ij})\times\mbb{C}_{q}^*$.

The Hamiltonian action given by $G_{\mbf{v}}$ gives the moment map:
\begin{align}
\mu:T^*\text{Rep}_{Q^{fr}}(\mbf{v},\mbf{w})\rightarrow\mf{g}_{\mbf{v}}^*
\end{align}

If we choose the corresponding algebraic group
\begin{align}
T_{\mbf{w}}=\mbb{C}^*\times\prod_{e\in E}\mbb{C}^*\times\prod_{i\in I}GL(W_i)\subset G_{edge}\times\prod_{i\in I}GL(W_i)
\end{align}

It acts on Nakajima quiver varieties as follows:
\begin{align}
(q,t_e,U_i)_{e\in E,i\in I}\cdot(X_e,Y_e,A_i,B_i)_{e\in E,i\in I}=(t_e^{-1}X_e,q^{-1}t_eY_e,A_iU_i^{-1},q^{-1}U_iB_i)_{e\in E,i\in I}
\end{align}

Given $\bm{\theta}\in\mbb{Z}^I$, it defines a character of $G_{\mbf{v}}$ by the convention
\begin{align}
(g_i)\mapsto\prod_{i\in I}(\text{det}(g_i))^{\theta_i}\in\mbb{C}^*
\end{align}

A Nakajima quiver variety with the stability condition $\bm{\theta}$ is defined as:
\begin{align}
M_{\bm{\theta},\zeta}(\mbf{v},\mbf{w})=\mu^{-1}(\zeta)//^{\bm{\theta}-ss}G_{\mbf{v}}
\end{align}

\begin{defn}
We say that $(\bm{\theta},\zeta)$ is \textbf{generic} if it is in the complement of the following hyperplane:
\begin{align}
\alpha_i\cdot\bm{\theta}=\alpha_i\cdot\zeta=0
\end{align}
for some collection of vectors $\{\alpha_i\}\in\mbb{N}^I$ stated in \cite{Na94}.
\end{defn}

For generic $\bm{\theta}\in\mbb{Z}^I$, it is a quasiprojective smooth variety. Here in this paper we don't give speicific choice of $\bm{\theta}$, we just require that $M_{\bm{\theta},\zeta}(\mbf{v},\mbf{w})$ is smooth. Furthermore we choose $\zeta=0$.

A good thing for the Nakajima quiver variety is that it satisfies the polarisation:
\begin{align}
TM_{\theta,0}(\mbf{v},\mbf{w})=T^{1/2}+q^{-1}\otimes(T^{1/2})^{\vee}
\end{align}

We will give several choices of the polarisation in the paper, for the detail see \ref{polarisation-0} and \ref{polarisaion-large-framing}.

If we fix the stability condition $\bm{\theta}$ for the quiver variety $M_{\bm{\theta},0}(\mbf{v},\mbf{w})$, we define $C_{ample}\in\text{Hom}_{Grp}(G_{\mbf{v}},\mbb{C}^*)\otimes\mbb{R}\cong\mbb{R}^{I}$ the ample cone to be the closure of $\mbb{R}_{>0}\bm{\theta}$ with the point $\bm{\theta}'$ such that the corresponding GIT quotient is isomorphic to the one of $\bm{\theta}$.

In this paper we will often use the notation $M(\mbf{v},\mbf{w})$ to stand for the quiver variety $M_{\bm{\theta},0}(\mbf{v},\mbf{w})$ with the chosen stability condition. 

\subsection{Equivariant $K$-theory of Nakajima quiver varieties}
We abbreviate $T_{\mbf{w}}=T$, so the equivariant $K$-theory $K_{T}(X)$ is a module over $K_{T}(pt)$ such that:
\begin{align}
K_{T}(pt)=\mbb{Z}[q^{\pm1},t_{e}^{\pm1},u_{ia}^{\pm1}]_{e\in E,i\in I,1\leq a\leq w_i}^{Sym}
\end{align}

The generators for the equivariant $K$-theory $K_{T}(M(\mbf{v},\mbf{w}))$ of Nakajima quiver varieties can be described via the Kirwan surjectivity \cite{MN18}.

Since $M_{\bm{\theta},0}(\mbf{v},\mbf{w})$ is a GIT quotient, we have a stack embedding:
\begin{align}
i:M_{\bm{\theta},\zeta}(\mbf{v},\mbf{w})\hookrightarrow[\mu^{-1}(0)/G_{\mbf{v}}]
\end{align}
This induces a pullback map on the equivariant $K$-theory:
\begin{align}
i^*:K_{T}([\mu^{-1}(0)/G_{\mbf{v}}])\rightarrow K_{T}(M_{\bm{\theta},\zeta}(\mbf{v},\mbf{w}))
\end{align}

It has been proved by McGerty and Nevins \cite{MN18} that this map is surjective. Since $K_{T}([\mu^{-1}(0)/G_{\mbf{v}}])\cong K_{T\times G_{\mbf{v}}}(pt)$, the equivariant $K$-theory $K_{T}([\mu^{-1}(0)/G_{\mbf{v}}])$ of Nakajima quiver varieties is generated by $K_{T}(pt)$ and its tautological bundles.

Let $(-,-)$ be a bilinear form on $K_{T}(X)$ defined by the following formula:
\begin{align}\label{twisted-bilinear-form}
(\mc{F},\mc{G})=\chi(\mc{F}\otimes\mc{G}\otimes\mc{K}_{X}^{-1/2})
\end{align}

Here $\mc{K}_{X}$ is the canonical bundle and $\chi=p_*:K_{T}(X)\rightarrow K_{T}(pt)$ is the equivariant Euler characteristic via the canonical map $p:X\rightarrow pt$. Since Nakajima quiver varieties admits the polarisation for its tangnet bundle, it implies that $\mc{K}_{X}$ admits its square root $\mc{K}_{X}^{1/2}$.

\subsection{K-theoretic stable envelope and geometric $R$-matrix}
We first review the definition of the $K$-theoretic stable envelopes for the quiver varieties.

Given $X:=M(\mbf{v},\mbf{w})$ a Nakajima quiver variety and $T_{\mbf{w}}$ acting on $M(\mbf{v},\mbf{w})$. Given a subtorus $A\subset T_{\mbf{w}}$ in the kernel of $q$. Choose a chamber $\mc{C}$ of the torus $A$ and a fractional line bundle $s\in\text{Pic}(X)\otimes\mbb{Q}$. By definition, the $K$-theoretic stable envelope is a unique $K$-theory class
\begin{align}
\text{Stab}_{\mc{C}}^{s}\subset K_{G}(X\times X^A)
\end{align}

supported on $\text{Attr}_{\mc{C}}^f$, such that it induces the morphism
\begin{align}
\text{Stab}_{\mc{C},s}:K_{G}(X^A)\rightarrow K_{G}(X)
\end{align}

such that if we write $X^T=\sqcup_{\alpha}F_{\alpha}$ into components:
\begin{itemize}
	\item The diagonal term is given by the structure sheaf of the attractive space:
	\begin{align}
	\text{Stab}_{\mc{C},s}|_{F_{\alpha}\times F_{\alpha}}=(-1)^{\text{rk }T_{>0}^{1/2}}(\frac{\text{det}(\mc{N}_{-})}{\text{det}T_{\neq0}^{1/2}})^{1/2}\otimes\mc{O}_{\text{Attr}}|_{F_{\alpha}\times F_{\alpha}}
	\end{align}

	\item The $T$-degree of the stable envelope has the bounding condition for $F_{\beta}\leq F_{\alpha}$:
	\begin{align}
	\text{deg}_{T}\text{Stab}_{\mc{C},s}|_{F_{\beta}\times F_{\alpha}}+\text{deg}_{T}s|_{F_{\alpha}}\subset\text{deg}_{T}\text{Stab}_{\mc{C},s}|_{F_{\beta}\times F_{\beta}}+\text{deg}_{T}s|_{F_{\beta}}
	\end{align}

	We require that for $F_{\beta}<F_{\alpha}$, the inclusion $\subset$ is strict.
\end{itemize}

Here note that the partial ordering on the components of the fixed point set coincides with "ample partial ordering". Given $\bm{\theta}\in\text{Pic}(M_{\bm{\theta},0}(\mbf{v},\mbf{w}))$ the image of the corresponding stability condition in the Picard group, and $\sigma\in\mc{C}$ a character of $A$ then:
\begin{align}
F_2\leq F_1\leftrightarrow\langle\bm{\theta}_{F_1},\sigma\rangle\leq\langle\bm{\theta}_{F_2},\sigma\rangle
\end{align}

Here given $F=M(\mbf{v}_1,\mbf{w}_1)\times M(\mbf{v}_2,\mbf{w}_2)$, the function $\langle\bm{\theta}_{F},\sigma\rangle$ is defined as:
\begin{align}
\langle\bm{\theta}_{F},\sigma\rangle=\langle\mbf{v}_1,\bm{\theta}_1\rangle\sigma_1+\langle\mbf{v}_2,\bm{\theta}_2\rangle\sigma_2
\end{align}
Here $\sigma_1,\sigma_2\in\mc{C}$ is chosen such that $\sigma_1-\sigma_2>0$.

The uniqueness and existence of the $K$-theoretic stable envelope was given in \cite{AO21} and \cite{O21}. In \cite{AO21}, the consturction is given by the abelinization of the quiver varieties. In \cite{O21}, the construction is given by the stratification of the complement of the attracting set, which is much more general.

Here we choose the following polarisation for point $\bm{\theta}\in C_{ample}$ in the ample cone given by the stability conditions of the quiver varieties:
\begin{equation}\label{polarisation-0}
T^{1/2}M(\mbf{v},\mbf{w})=\sum_{e=ij\in E}\frac{\mc{V}_j}{t_e\mc{V}_i}-\sum_{i\in I}\frac{\mc{V}_i}{\mc{V}_i}+\sum_{i\in I,\theta_i>0}\frac{\mc{V}_i}{\mc{W}_i}+\sum_{i\in I,\theta_i<0}\frac{\mc{W}_i}{q\mc{V}_i}
\end{equation}

\begin{equation}\label{polarisation-1}
\begin{aligned}
\mc{N}^+_{M(\mbf{v}_1,\mbf{w}_1)\times M(\mbf{v}_2,\mbf{w}_2)}=&\sum_{ij\in E}(\frac{\mc{V}_j'}{t_e\mc{V}_i''}+\frac{t_e\mc{V}_i'}{q\mc{V}_j''})-\sum_{i\in I}(1+q^{-1})(\frac{\mc{V}_i'}{\mc{V}_i''})+\sum_{i\in I,\theta_i>0}(\frac{\mc{V}_i'}{\mc{W}_i''}+\frac{\mc{W}_i'}{q\mc{V}_i''})+\\
&+\sum_{i\in I,\theta_i<0}(\frac{\mc{V}_i''}{\mc{W}_i'}+\frac{\mc{W}_i''}{q\mc{V}_i'})
\end{aligned}
\end{equation}
\begin{align}\label{polarisation-2}
T^{1/2}_{-}=\sum_{ij\in E}\frac{\mc{V}_j''}{t_e\mc{V}_i'}-\sum_{i\in I}\frac{\mc{V}_i''}{\mc{V}_i'}+\sum_{i\in I,\theta_i>0}\frac{\mc{V}_i''}{W_{i}'}+\sum_{i\in I,\theta_i<0}q^{-1}\frac{\mc{W}_i'}{\mc{V}_i''}
\end{align}

We will often use the following notation:
\begin{equation}\label{difference}
\Delta\mbf{w}:=\mbf{w}^{>0}-\mbf{w}^{<0}=\sum_{i\in I,\theta_i>0}\text{dim}(W_i)\mbf{e}_i-\sum_{i\in I,\theta_i<0}\text{dim}(W_i)\mbf{e}_i
\end{equation}

The stable envelope has the factorisation property called the triangle lemma \cite{O15}. Given a subtorus $A'\subset A$ with the corresponding chamber $\mc{C}_{A'},\mc{C}_{A}$, we have the following diagram commute:
\begin{equation}\label{triangle-lemma}
\begin{tikzcd}
K_{G}(X^A)\arrow[rr,"\text{Stab}_{\mc{C}_A,s}"]\arrow[dr,"\text{Stab}_{\mc{C}_{A/A'},s}"]&&K_{G}(X)\\
&K_{G}(X^{A'})\arrow[ur,"\text{Stab}_{\mc{C}_{A'},s}"]&
\end{tikzcd}
\end{equation}

In this paper, we will always choose $s$ infinitesiamlly close to $0$ on the ample cone $C_{ample}$. 

Let us focus on the case of the quiver varieties $M(\mbf{v},\mbf{w})$. Choose the framing torus $\sigma:\mbb{C}^*\rightarrow A_{\mbf{w}}\subset G_{\mbf{w}}$ such that:
\begin{align}
\mbf{w}=a_1\mbf{w}_1+\cdots+a_k\mbf{w}_k
\end{align}

In this case the fixed point is given by:
\begin{align}
M(\mbf{v},\mbf{w})^{\sigma}=\bigsqcup_{\mbf{v}_1+\cdots+\mbf{v}_k=\mbf{v}}M(\mbf{v}_1,\mbf{w}_1)\times\cdots\times M(\mbf{v}_k,\mbf{w}_k)
\end{align}

Denote $K(\mbf{w}):=\bigoplus_{\mbf{v}}K_{G_{\mbf{w}}}(M(\mbf{v},\mbf{w}))$, it is easy to see that the stable envelope $\text{Stab}_{s}$ gives the map:
\begin{align}
\text{Stab}_{\mc{C},s}:K(\mbf{w}_1)\otimes\cdots\otimes K(\mbf{w}_k)\rightarrow K(\mbf{w}_1+\cdots+\mbf{w}_k)
\end{align}

Using the $K$-theoretic stable envelope, we can define the geometric $R$-matrix as:
\begin{align}
\mc{R}^{s}_{\mc{C}}:=\text{Stab}_{-\mc{C},s}^{-1}\circ\text{Stab}_{\mc{C},s}:K(\mbf{w}_1)\otimes\cdots\otimes K(\mbf{w}_k)\rightarrow K(\mbf{w}_1)\otimes\cdots\otimes K(\mbf{w}_k)
\end{align}

Written in the component of the weight subspaces, the geometric $R$-matrix can be written as:
\begin{equation}
\begin{aligned}
&\mc{R}^{s}_{\mc{C}}:=\text{Stab}_{-\mc{C},s}^{-1}\circ\text{Stab}_{\mc{C},s}:\bigoplus_{\mbf{v}_1+\cdots+\mbf{v}_k=\mbf{v}}K(\mbf{v}_1,\mbf{w}_1)\otimes\cdots\otimes K(\mbf{v}_k,\mbf{w}_k)\\
&\rightarrow \bigoplus_{\mbf{v}_1+\cdots+\mbf{v}_k=\mbf{v}}K(\mbf{v}_1,\mbf{w}_1)\otimes\cdots\otimes K(\mbf{v}_k,\mbf{w}_k)
\end{aligned}
\end{equation}

From the triangle diagram \ref{triangle-lemma} of the stable envelope, we can further factorise the geometric $R$-matrix into the smaller parts:
\begin{align}\label{abstract-decomposition}
\mc{R}^s_{\mc{C}}=\prod_{1\leq i<j\leq k}\mc{R}^s_{\mc{C}_{ij}}(\frac{a_i}{a_j}),\qquad \mc{R}^s_{\mc{C}_{ij}}(\frac{a_i}{a_j}):K(\mbf{w}_i)\otimes K(\mbf{w}_j)\rightarrow K(\mbf{w}_i)\otimes K(\mbf{w}_j)
\end{align}

Each $\mc{R}^{s}_{\mc{C}_{ij}}(u)$ satisfies the trigonometric Yang-Baxter equation with the spectral parametres.

In the language of the integrable model, we denote $V_{i}(u_i)$ as the modules of type $K(\mbf{w})$ defined above with the spectral parametre $u_i$. The formula \ref{abstract-decomposition} means that:
\begin{align}
\mc{R}^s_{\bigotimes^{\leftarrow}_{i\in I}V_i(a_i),\bigotimes^{\leftarrow}_{j\in J}V_j(a_j)}=\prod_{i\in I}^{\rightarrow}\prod_{j\in J}^{\leftarrow}\mc{R}^{s}_{V_i,V_j}(\frac{a_i}{a_j}):
\end{align}

The $K$-theoretical stable envelope $\text{Stab}_{\mc{C},s}$ is locally constant on $s$ if and only if $s$ crosses the following hyperplanes.
\begin{align}
w=\{s\in\text{Pic}(X)\otimes\mbb{R}|(s,\alpha)+n=0,\forall\alpha\in\text{Pic}(X)\}\subset\text{Pic}(X)\otimes\mbb{R}
\end{align}

Now fix the slope $\mbf{m}$ and the cocharacter $\sigma:\mbb{C}^*\rightarrow A_{\mbf{w}}$. We choose an ample line bundle $\mc{L}\in\text{Pic}(X)$ with $X=M(\mbf{v},\mbf{w}_1+\mbf{w}_2)$ and a suitable small number $\epsilon$ such that $\mbf{m}$ and $\mbf{m}+\epsilon\mc{L}$ are separated by just one wall $w$, we define the \textbf{wall $R$-matrices} as:
\begin{align}
R_{w}^{\pm}:=\text{Stab}_{\sigma,\mbf{m}+\epsilon\mc{L}}^{-1}\circ\text{Stab}_{\sigma,\mbf{m}}:\bigoplus_{\mbf{v}_1+\mbf{v}_2=\mbf{v}}K(\mbf{v}_1,\mbf{w}_1)\otimes K(\mbf{v}_2,\mbf{w}_2)\rightarrow\bigoplus_{\mbf{v}_1+\mbf{v}_2=\mbf{v}}K(\mbf{v}_1,\mbf{w}_1)\otimes K(\mbf{v}_2,\mbf{w}_2)
\end{align}

It is an integral $K$-theory class in $K_{T}(X^A\times X^A)$. Note that the choice of $\epsilon$ depends on $M(\mbf{v},\mbf{w}_1+\mbf{w}_2)$ just to make sure that there is only one wall between $\mbf{m}$ and $\mbf{m}+\epsilon\mc{L}$ corresponding to $w$. 

We use the notation that given $A=A_{\alpha}\in\text{End}(K(\mbf{w}_1)\otimes K(\mbf{w}_2))$, $A_{\alpha}:K_{T}(M(\mbf{v}_1,\mbf{w}_1))\otimes K_{T}(M(\mbf{v}_2,\mbf{w}_2))\rightarrow K_{T}(M(\mbf{v}_1-\alpha,\mbf{w}_1))\otimes K_{T}(M(\mbf{v}_2+\alpha,\mbf{w}_2))$.
By definition it is easy to see that $R_{w}^{+}=1+\sum_{\alpha}R_{w,\alpha}^{+}$ is upper-triangular in the sense that $\bm{\theta}\cdot\alpha>0$, and $R_{w}^{-}=1+\sum_{\alpha}R_{w,\alpha}^{-}$ is lower triangular in the sense that $\bm{\theta}\cdot\alpha<0$.

Now choose the torus $T$ such that $M(\mbf{v},\mbf{w})^T=\bigoplus_{\mbf{v}_1+\mbf{v}_2=\mbf{v}} M(\mbf{v}_1,\mbf{w}_1)\times M(\mbf{v}_2,\mbf{w}_2)$ and $\mbf{w}=\mbf{w}_1+a\mbf{w}_2$. It has also been proved in \cite{OS22} that the wall $R$-matrices are monomial in spectral parametre $u$:
\begin{align}
R_{w}^{\pm}|_{F_2\times F_1}=
\begin{cases}
1&F_1=F_2\\
(\cdots)a^{\langle\mu(F_2)-\mu(F_1),\mbf{m}\rangle}&F_1\geq F_2\text{ or }F_1\leq F_2\\
0&\text{Otherwise}
\end{cases}
\end{align}

Here $\mu$ is a locally constant map $\mu:X^A\rightarrow H_2(X,\mbb{Z})\otimes A^{\wedge}$ defined up to an overall translation such that $\mu(F_1)-\mu(F_2)=[C]\otimes v$ with $C$ an irreducible curve joining $F_1$ and $F_2$ with tangent weight $v$ at $F_1$. Usually it is convenient for us to choose $A$ to be one-dimensional torus such that $A^{\wedge}\cong\mbb{Z}$. $(\cdots)$ stands for the term in $K_{T}(pt)$.

In the case of the wall $R$-matrices $R_{w}^{\pm}$, $\mu(F_2)-\mu(F_1)$ corresponds to $\pm k\alpha$ with $\alpha$ being the root corresponding to the wall $w$.

Fix the stable envelope $\text{Stab}_{\sigma,\mbf{m}}$ and $\text{Stab}_{\sigma,\infty}$, we can have the following factorisation of $\text{Stab}_{\pm,\mbf{m}}$ near $u=0,\infty$:
\begin{equation}
\begin{aligned}
\text{Stab}_{\sigma,\mbf{m}}=&\text{Stab}_{\sigma,-\infty}\cdots\text{Stab}_{\sigma,\mbf{m}_{-2}}\text{Stab}_{\sigma,\mbf{m}_{-2}}^{-1}\text{Stab}_{\sigma,\mbf{m}_{-1}}\text{Stab}_{\sigma,\mbf{m}_{-1}}^{-1}\text{Stab}_{\sigma,\mbf{m}}\\
=&\text{Stab}_{\sigma,-\infty}\cdots R_{\mbf{m}_{-2},\mbf{m}_{-1}}^+R_{\mbf{m}_{-1},\mbf{m}}^+
\end{aligned}
\end{equation}

\begin{equation}
\begin{aligned}
\text{Stab}_{-\sigma,\mbf{m}}=&\text{Stab}_{-\sigma,\infty}\cdots\text{Stab}_{-\sigma,\mbf{m}_2}\text{Stab}_{-\sigma,\mbf{m}_2}^{-1}\text{Stab}_{-\sigma,\mbf{m}_1}\text{Stab}_{-\sigma,\mbf{m}_1}^{-1}\text{Stab}_{-\sigma,\mbf{m}}\\
=&\text{Stab}_{-\sigma,\infty}\cdots R_{\mbf{m}_2,\mbf{m}_1}^{-}R_{\mbf{m}_1,\mbf{m}}^{-}
\end{aligned}
\end{equation}

Here $R_{\mbf{m}_1,\mbf{m}_2}^{\pm}=\text{Stab}_{\pm\sigma,\mbf{m}_1}^{-1}\text{Stab}_{\pm\sigma,\mbf{m}_2}$ is the wall $R$-matrix. For simplicity we always choose generic slope points $\mbf{m}_i$ such that there is only one wall between $\mbf{m}_1$ and $\mbf{m}_2$. In this case we use $R_{w}^{\pm}$ as $R_{\mbf{m}_1,\mbf{m}_2}^{\pm}$. Note that this notation does not mean that $R_{w}$ only depends on the wall $w$, but we still use the notation for simplicity.

This gives the KT factorisation of the geometric $R$-matrix:
\begin{align}\label{factorisation-geometry}
\mc{R}^{s}(u)=\prod_{i<0}^{\leftarrow}R_{w_i}^{-}R_{\infty}\prod_{i\geq0}^{\leftarrow}R_{w_i}^{+}
\end{align}

This factorisation is well-defined in the topology of the Laurent formal power series in the spectral parametre $a$.

\section{\textbf{Quasimap moduli space and vertex functions}}
\subsection{Quasimap moduli space}\label{quasimap-moduli-space}
A quasimap $f:C\rightarrow X$ with a domain $C\cong\mbb{P}^1$ to a Nakajima quiver variety $X=M_{\bm{\theta},0}(\mbf{v},\mbf{w})$ is defined by the following data:
\begin{itemize}
	\item A collection of vector bundles $\mc{V}_i$, $i\in I$ on $C$ with ranks $\mbf{v}_i$.
	\item A collection of trivial vector bundle $\mc{Q}_{ij}$ and $\mc{W}_{i}$, $i,j\in I$ on $C$ with ranks $m_{ij}$ and $\mbf{w}_i$ respectively.
	\item A section
	\begin{align}
	f\in H^0(C,\mc{M}\oplus\mc{M}^*\otimes q^{-1})
	\end{align}
	satisfying the moment map condition $\mu=0$, where
	\begin{align}
	\mc{M}=\bigoplus_{ij\in E}\text{Hom}(\mc{V}_i,\mc{V}_j)\otimes\mc{Q}_{ij}\oplus\bigoplus_{i\in I}\text{Hom}(\mc{W}_i,\mc{V}_i)
	\end{align}
	Here $q^{-1}$ stands for a trivial line bundle on $C$ with $G_{\mbf{v}}$-equivariant weight $q^{-1}$.
\end{itemize}
The degree of a quasimap is defined as $\mbf{d}=(\text{deg}(\mc{V}_i))_{i=1}^n\in\text{Pic}(M(\mbf{v},\mbf{w}))\cong\mbb{Z}^n$.

Similarly, the quasimap to $M(\mbf{v}_1,\mbf{w}_1)\times\cdots\times M(\mbf{v}_k,\mbf{w}_k)$ is defined by the following data:
\begin{itemize}
	\item A collection of vector bundles $(\mc{V}_{i}^{(\alpha)})$, $i\in I$, $\alpha=1,\cdots,k$ of rank $\mbf{v}_i^{(\alpha)}$.
	\item A collection of trivial vector bundle $\mc{Q}_{ij}^{(\alpha)}$ and $\mc{W}_{i}^{(\alpha)}$ on $C$ with ranks $m_{ij}^{(\alpha)}$ and $\mbf{w}_{i}^{(\alpha)}$ respectively.
	\item A section
	\begin{align}
	f\in \bigoplus_{\alpha}H^0(C,\mc{M}^{(\alpha)}\oplus(\mc{M}^{(\alpha)})^*\otimes q^{-1})
	\end{align}
	satisfying the moment map condition $\mu=0$, where
	\begin{align}
	\mc{M}^{(\alpha)}=\bigoplus_{ij\in E}\text{Hom}(\mc{V}_i^{(\alpha)},\mc{V}_j^{(\alpha)})\otimes\mc{Q}_{ij}^{(\alpha)}\oplus\bigoplus_{i\in I}\text{Hom}(\mc{W}_i^{(\alpha)},\mc{V}_i^{(\alpha)})
	\end{align}
\end{itemize}

It is worth noting that in this case the degree of a quasimap is defined as $\{\mbf{d}^{(\alpha)}\}=(\text{deg}(\mc{V}_i^{(\alpha)}))_{i=1}^n\in\text{Pic}(M(\mbf{v}_1,\mbf{w}_1))\oplus\cdots\oplus\text{Pic}(M(\mbf{v}_k,\mbf{w}_k))$.

A stable quasimap from a genus $0$ curve $C$ to $M_{\bm{\theta},0}(\mbf{v},\mbf{w})\subset[\mu^{-1}(0)/G_{\mbf{v}}]$ \textbf{relative} to points $p_1,\cdots,p_m\in C$ is given with the curve $C'$, which has at most nodal singularities, $p_i'$ the nonsingular points of $C'$ and a regular map $\pi:C'\rightarrow C$ such that besides the data above replacing $C$ by $C'$. It satisfies the following conditions:
\begin{enumerate}
	\item There is a distinguished component $C_0'$ of $C'$ such that $\pi|_{C_0'}:C_0'\cong C$, and $\pi(C'\backslash C_0')$ is zero-dimensional.
	\item $\pi(p_i')=p_i$.
	\item $f(p)$ is stable i.e. $f(p)\in M_{\bm{\theta},0}(\mbf{v},\mbf{w})$ for all $p\in C'\backslash B$ where $B$ is finite subset of $C'$.
	\item The set $B$ is disjoint from the nodes and points $p_1',\cdots,p_m'$.
	\item Let $\tilde{C}=\overline{C'\backslash C_0}$ be the closure of $C'\backslash C_0$, and let $q_i$ be the nodes $C_0\cap\tilde{C}$. The fractional line bundle $\omega_{\tilde{C}}(\sum_{i}p_i'+\sum_{j}q_j)\otimes\mc{L}_{\bm{\theta}}^{\epsilon}$ is ample for every rational $\epsilon>0$ where $\mc{L}_{\bm{\theta}}:=P\times_{G}\mbb{C}_{\bm{\theta}}$
\end{enumerate}

A quasimap $f$ is called \textbf{nonsingular} at $p$ if $f(p)\subset M_{\bm{\theta},0}(\mbf{v},\mbf{w})$ and the quasimap is not relative to $p$.

Let $\text{QM}^{\mbf{d}}_{\substack{\text{rel }p_1,\cdots,p_m\\\text{nonsing }q_1,\cdots,q_n}}$ denote the stack parametrising stable genus zero quasimaps relative to $p_1,\cdots,p_m$ and nonsingular at $q_1,\cdots,q_n$ of fixed degree $\mbf{d}$.  

The moduli space of quasimaps constructed above has perfect deformation-obstruction theory \cite{CKM14}. This allows us to construct a tangent virtual bundle $T^{vir}$, a virtual structure sheaf $\mc{O}_{vir}$ and a virtual canonical bundle $\mc{K}_{vir}$. For the Nakajima quiver variety, it admits polarisation for the tangent bundle $TX$:
\begin{align}
TX=T^{1/2}X+q^{-1}(T^{1/2}X)^{\vee}\in K(X)
\end{align}
The choice of the polarisation induce a virtual bundle $\mc{T}^{1/2}$ over $C\times\text{QM}^{\mbf{d}}(X)$.

We will focus on the quasimap moduli space $QM^{\mbf{d}}_{(\cdots)p_1,(\cdots)p_2}(X)$ with two marked points $p_1,p_2$, i.e. the condition on $p_1,p_2$ can be either nonsingular or relative.

For the virtual tangent bundle on $QM^{\mbf{d}}_{p_1,p_2}(X)$ with $X$ the Nakajima quiver variety , we have that:
\begin{align}
T_{vir}=\text{Def}-\text{Obs}=H^*(C,\mc{T}),\qquad\mc{T}=\mc{M}\oplus q^{-1}\mc{M}^*-(1+q^{-1})\bigoplus_{(i)\in I}\mc{End}(\mc{V}_{i})
\end{align}

In the case of the product of the Nakajima quiver varieties $X_1\times\cdots\times X_n$, we have that the virtual tangent bundle bieng written as
\begin{align}
T_{vir}=\text{Def}-\text{Obs}=H^*(C,\mc{T}),\qquad\mc{T}=\sum_{\alpha}\mc{M}^{(\alpha)}\oplus q^{-1}(\mc{M}^{(\alpha)})^*-(1+q^{-1})\bigoplus_{(i)\in I}\mc{End}(\mc{V}_{i,(\alpha)})
\end{align}
This implies that
\begin{align}
T_{vir}QM(X_1\times\cdots\times X_n)=\sum_{i=1}^nT_{vir}QM(X_i)
\end{align}

In our situation, we tend to choose the symmetrise virtual structure sheaf:
\begin{align}
\hat{\mc{O}}_{vir}:=\mc{O}_{vir}\otimes(\mc{K}_{vir}\frac{\text{det}(\mc{T}^{1/2})|_{p_2}}{\text{det}(\mc{T}^{1/2})|_{p_1}})^{1/2}
\end{align}
and this is the reason why we need to choose the shift for the bilinear form of the $K$-theory in the formula \ref{twisted-bilinear-form}.

\subsection{Vertex functions with descendents}
Let $\text{QM}^{\mbf{d}}(M(\mbf{v},\mbf{w}))$ be the moduli space of stable quasimaps from $\mbb{P}^1$ to $M_{\bm{\theta},0}(\mbf{v},\mbf{w})$. There is a one-dimensional torus $\mbb{C}^*_{p}$ over $\mbb{P}^1$. The fixed point set of this action consist of two points $\{p_1,p_2\}=\{0,\infty\}\subset\mbb{P}^1$. 

Given a point $p\in\mbb{P}^1$, let $\text{QM}^{\mbf{d}}_{\text{nonsing }p}(M(\mbf{v},\mbf{w}))\subset\text{QM}^{\mbf{d}}(M(\mbf{v},\mbf{w}))$ be the open subset of quasimaps nonsingular at $p$. This subset comes together with the evaluation map:
\begin{align}
\text{ev}_{p}:\text{QM}^{\mbf{d}}_{\text{nonsing }p}(M(\mbf{v},\mbf{w}))\rightarrow M_{\bm{\theta},0}(\mbf{v},\mbf{w})
\end{align}
sending a quasimap to its value at $p$. This map is not a proper map, but it becomes proper when restricted to the $\mbb{C}_{p}^*$-fixed point part:
\begin{align}
\text{ev}_p:(\text{QM}^{\mbf{d}}_{\text{nonsing }p}(M_{\bm{\theta},0}(\mbf{v},\mbf{w})))^{\mbb{C}_p^*}\rightarrow M_{\bm{\theta},0}(\mbf{v},\mbf{w})
\end{align}

Moreover, the fixed point part $\text{QM}^{\mbf{d}}_{\text{nonsing }p}(M_{\bm{\theta},0}(\mbf{v},\mbf{w})))^{\mbb{C}_p^*}$ is empty if $\mbf{d}\notin C_{ample}^{\vee}$.

The moduli space of relative quasimaps $\text{QM}^{\mbf{d}}_{\text{rel }p}(M_{\theta,0}(\mbf{v},\mbf{w}))$ is a resolution of the map $\text{ev}$ such that we have a commutative diagram:
\begin{equation}
\begin{tikzcd}
 &\text{QM}^{\mbf{d}}_{\text{rel }p}(M_{\theta,0}(\mbf{v},\mbf{w}))\arrow[dr,"\hat{\text{ev}}_{p}"]&\\
\text{QM}^{\mbf{d}}_{p}(M_{\theta,0}(\mbf{v},\mbf{w}))\arrow[ur]\arrow[rr,"\text{ev}_{p}"]&&M_{\theta,0}(\mbf{v},\mbf{w})
\end{tikzcd}
\end{equation}

Now choose a tautological class $\tau\in K([\mu^{-1}(0)/G_{\mbf{v}}])$ in the moduli stack. We define the \textbf{bare vertex }with a descendent $\tau$ as the following formal power series:
\begin{align}
V^{(\tau)}(z)=\sum_{\mbf{d}\in C_{ample}^{\vee}}z^{\mbf{d}}\text{ev}_{p_2,*}(\hat{\mc{O}}_{\text{vir}}^{\mbf{d}}\otimes(\tau(\mc{V}))|_{p_1})
\in K_{T\times\mbb{C}_p^*}(M_{\bm{\theta},0}(\mbf{v},\mbf{w}))_{loc}\otimes\mbb{Q}[[z^{\mbf{d}}]]
\end{align}

We can compute the vertex function via equivariant localisation, it is important to understand the structure of the fixed locus $(\text{QM}_{\text{nonsing }\infty})^{\mbb{C}_p^*}$. The data for the fixed loci can be described as follows: Define
\begin{align}
\mbf{V}_a=\bigoplus_{k\in\mbb{Z}}\mbf{V}_a[k]=H^0(\mc{V}_a|_{C\backslash\{p_2\}})
\end{align}
where $\mbf{V}_a[k]$ is the subspace of weight $k$ with respect to $\mbb{C}^*_p$. By invariance, all quiver maps preserve this weight decomposition. We define the framing spaces $\mbf{W}[k]$ in the same way and obtain
\begin{align}\label{frame-grading}
\mbf{W}_a[k]=
\begin{cases}
W_a&k\leq0\\
0&k>0
\end{cases}
\end{align}
since the bundles $\mc{W}_a$ are trivial.

Multiplication by the coordinate induces an embedding
\begin{align}
\mbf{V}_a[k]\hookrightarrow\mbf{V}_a[k-1]\hookrightarrow\cdots\hookrightarrow\mbf{V}_a[-\infty]
\end{align}
compatible with the quiver maps, where $V_a$ is the quiver data for the point $f(\infty)\in X$. A $\mbb{C}_p^*$-fixed stable quasimap $f$ takes a constant stable value on $C\backslash\{0,\infty\}$ and, since $f$ is nonsingular at infinity, $f(\infty)$ is that generic value of $f$.

We conclude:
\begin{align}
\text{QM}_{\text{nonsing }\infty}^{\mbb{C}_p^*}(M(\mbf{v},\mbf{w}))=\begin{pmatrix}\text{A }\bm{\theta}\text{-stable quiver representation}\\+\text{A flag of subrepresentations satisfying \ref{frame-grading}}\end{pmatrix}/\prod_{a\in I}GL(V_a)
\end{align}

Similarly, for the space of quasimaps to $X_1\times\cdots\times X_k=M(\mbf{v}_1,\mbf{w}_1)\times\cdots\times M(\mbf{v}_k,\mbf{w}_k)$. The corresponding $\mbb{C}_p^*$-fixed point can be descirbed in the following way: We keep the convention $\mbf{V}_{a}$ and $\mbf{W}_{a}$ as above. In this case, the space $\mbf{V}_{a}$ and $\mbf{W}_{a}$ admits another factorisation:
\begin{align}\label{fixed-decomposition}
\mbf{V}_a=\bigoplus_{i=1}^k\mbf{V}_a^{(i)},\qquad\mbf{W}_a=\bigoplus_{i=1}^{k}\mbf{W}_a^{(i)}
\end{align}
with respect to the group action $G_{\mbf{v}_1}\times\cdots\times G_{\mbf{v}_k}$ and $G_{\mbf{w}_1}\times\cdots\times G_{\mbf{w}_k}$ which is compatible with the action given by $\mbb{C}_{p}^*$. Thus the corresponding fixed point is given as:
\begin{equation}
\begin{aligned}
&\text{QM}_{\text{nonsing }\infty}^{\mbb{C}_p^*}(M(\mbf{v}_1,\mbf{w}_1)\times\cdots\times M(\mbf{v}_k,\mbf{w}_k))\\
=&\begin{pmatrix}\text{A }\bm{\theta}\text{-stable quiver representation}\\+\text{A flag of subrepresentations satisfying \ref{frame-grading} and \ref{fixed-decomposition}}\end{pmatrix}/\prod_{i=1}^{k}G_{\mbf{v}_i}
\end{aligned}
\end{equation}

By the definition of the bare vertex function $\text{Vertex}^{(\tau)}(z)$, it is easy to see that it has the translation symmetry with respect to the line bundle class $\mc{L}$:
\begin{equation}\label{translation-symmetry}
\text{Vertex}^{(\tau\otimes\mc{L})}(a,z)=\mc{L}\text{Vertex}^{(\tau)}(a,zp^{\mc{L}})
\end{equation}

\subsection{Index limit of the vertex function}
Now we choose a subtorus $A\subset T$ for the quiver variety $X=M(\mbf{v},\mbf{w})$. Denote $X^A$ as the $A$-fixed point of $X$ and $i:X^A\hookrightarrow X$ the closed embedding. Also we choose the tautological class $\tau\in K_{T}(pt)[\cdots,x_{i1},\cdots,x_{in_i},\cdots]^{\text{Sym}}$ as the symmetric polynomial.

It has been proved in section $7.3$ of \cite{O15} that the vertex function with descendents has the index limit as $a\rightarrow0$:
\begin{align}\label{index-limit}
\lim_{a\rightarrow0}\text{Vertex}^{(\tau)}_{X}(a,z)=q^{\Omega}(-q^{-1/2})^{\text{deg }\mc{N}_{>0}}p^{-\text{deg }\mc{T}_{<0}^{1/2}}\text{Vertex}_{X^A}^{i^*\tau}(z_1,z_2)|_{z_1=z_2=z}
\end{align}
where
\begin{align}
\mc{N}_{>0}=\mc{T}_{>0}^{1/2}\oplus q^{-1}(\mc{T}_{<0}^{1/2})^{\vee}
\end{align}

and $\Omega$ is the operator $\Omega:=\frac{\text{codim}}{4}$

For example, for the case of Nakajima quiver varieties $M(\mbf{v},\mbf{w})$, choose a torus $A$ such that $\mbf{w}=\mbf{w}_1+a\mbf{w}_2$. In this case we have that:
\begin{align}\label{index-limit-formula}
\lim_{a\rightarrow0}\text{Vertex}^{(\tau(\mc{V}))}_{M(\mbf{v},\mbf{w})}(a,z)=q^{\frac{\text{codim}}{4}}(-q^{-1/2})^{\text{deg }\mc{N}_{>0}}p^{-\text{deg }\mc{T}_{<0}^{1/2}}\text{Vertex}_{X^A}^{i^*\tau(\mc{V})}(z_1,z_2)|_{z_1=z_2=z}
\end{align}

Using the polarisation \ref{polarisation-1} and \ref{polarisation-2}, we can see that the right hand side of \ref{index-limit-formula} can be further written as:
\begin{align}
q^{\frac{codim}{4}}\text{Vertex}_{X^A}^{i^*\tau(\mc{V})}(-z_1q^{-\frac{C\mbf{v}_2-\Delta\mbf{w}_2}{2}},-z_2q^{\frac{C\mbf{v}_1-\Delta\mbf{w}_1}{2}}p^{C\mbf{v}_1-\Delta\mbf{w}_1})|_{z_1=z_2=z}
\end{align}

Here $\Delta\mbf{w}$ is defined in \ref{difference}.

For simplicity, we usually use the coordinate transformation $z_1\mapsto-z_1$, $z_2\mapsto-z_2$. In fact this will not affect our result since without doing the transformation, we just need to consider the capping operator of the form $\Psi^{(\mbf{w}_1)}(-zq^{(\cdots)})\otimes\Psi^{(\mbf{w}_2)}(-zq^{(\cdots)})$.

Similarly, if we take $a\rightarrow\infty$, it is equivalent to take the $\mbf{w}=a^{-1}\mbf{w}_1+\mbf{w}_2$.

Now for the fully generalisation, we take:
\begin{align}
\mbf{w}=\mbf{w}_0+a_1\mbf{w}_1+\cdots+a_k\mbf{w}_k
\end{align}
such that $a_i$ goes whether to $0$ or to $\infty$.

\subsection{Factorisation property for the vertex functions}
Since the vertex function $\text{Vertex}^{(\tau)}_{X}(z)\in K_{T}(X)_{loc}$ is defined over the localised $K$-theory of $X$. It requires the analysis of the fixed point set $\text{QM}_{\text{nonsing }\infty}^{\mbb{C}_p^*}(M(\mbf{v},\mbf{w}))$. Moreover, as explained in section 3.3.3 of \cite{AFO18}, we can prove that $\text{QM}_{\text{nonsing }\infty}^{\mbb{C}_p^*}(M(\mbf{v},\mbf{w}))$ is a GIT quotient of the following stack quotient:
\begin{align}
\begin{pmatrix}\text{A }\bm{\theta}\text{-stable quiver representation}\\+\text{A flag of subrepresentations in }V_a\end{pmatrix}/\prod_{a\in I}GL(V_a)
\end{align}

Now we let a subtorus $A\subset T_{\mbf{w}}$ acts on the quiver varieties $M_{\bm{\theta},0}(\mbf{v},\mbf{w})$. Let $X^A=\bigsqcup_{\alpha}F_{\alpha}$. For every component $F_{\alpha}\subset X^A$, there is an obvious inclusion map:
\begin{align}
i_{QM}:QM(F_{\alpha})\hookrightarrow QM(X)^A
\end{align}

\begin{prop}\label{factorisation-vertex}
Given the GIT quotient $X=X_1\times\cdots\times X_k=M(\mbf{v}_1,\mbf{w}_1)\times\cdots\times M(\mbf{v}_k,\mbf{w}_k)$ of products of quiver varieties, we have:
\begin{align}
\text{Vertex}^{(\mc{F}_1\boxtimes\cdots\boxtimes\mc{F}_k)}_{X_1\times\cdots\times X_k}(z)=\text{Vertex}^{(\mc{F}_1)}_{X_1}(z)\otimes\cdots\otimes\text{Vertex}^{(\mc{F}_k)}_{X_k}(z)
\end{align}
\end{prop}
\begin{proof}
The vertex function $\text{Vertex}^{(\mc{F}_1\boxtimes\cdots\boxtimes\mc{F}_k)}_{X_1\times\cdots\times X_k}(z)$ can be computed using the localisation formula and the quiver description for the fixed point subset $\text{QM}_{\text{nonsing }\infty}^{\mbb{C}_p^*}(M(\mbf{v}_1,\mbf{w}_1)\times\cdots\times M(\mbf{v}_k,\mbf{w}_k))$. By the similar argument in \cite{AFO18}, one can also show that $\text{QM}_{\text{nonsing }\infty}^{\mbb{C}_p^*}(M(\mbf{v}_1,\mbf{w}_1)\times\cdots\times M(\mbf{v}_k,\mbf{w}_k))$ is a GIT quotient of the followint stack quotient:
\begin{align}
&\begin{pmatrix}\text{A }\bm{\theta}\text{-stable quiver representation}\\+\text{A flag of subrepresentations in }V_a\text{ satisfying \ref{fixed-decomposition}}\end{pmatrix}/\prod_{i=1}^{k}G_{\mbf{v}_i}\\
\cong&\prod_{i=1}^{k}\begin{pmatrix}\text{A }\bm{\theta}\text{-stable quiver representation}\\+\text{A flag of subrepresentations in }V_a^{(i)}\}\end{pmatrix}/G_{\mbf{v}_i}
\end{align}
that it becomes the product of GIT quotients

After the localisation, only the $(\mc{O}_{vir}\otimes\mc{K}_{vir})_{moving}$ part of $\mc{O}_{vir}\otimes\mc{K}_{vir}$ contributes to the computation. Using the result of section 3.4 in \cite{AFO18}:
\begin{align}
(\mc{O}_{vir}\otimes\mc{K}_{vir})_{moving}=\wedge^*(T_{vir,moving}^{X_1\times\cdots\times X_k})^{-1}
\end{align}

From the formula we can see that it remains to prove that:
\begin{align}\label{facto-1}
\wedge^*(T_{vir,moving}^{X_1\times\cdots\times X_k})=\wedge^*(T_{vir,moving}^{X_1})\cdots\wedge^*(T_{vir,moving}^{X_k})
\end{align}

\begin{align}\label{facto-2}
\text{deg}(\mc{T}^{1/2}_{vir,X_1\times\cdots\times X_k})=\sum_{i=1}^{k}\text{deg}(\mc{T}^{1/2}_{vir,X_i})
\end{align}

While the identities \ref{facto-1} \ref{facto-2} come from the definition and the fact that $\mc{T}^{1/2}_{X_1\times\cdots\times X_k}=\sum_{i=1}^k\mc{T}^{1/2}_{X_i}$. Thus the proof is finished.

\end{proof}

\textbf{Remark.}
The matrix coefficients of the vertex function can be computed using the method in the Appendix of \cite{AFO18} to do the computation. Namely, for each $G_{\mbf{v}_i}$, we choose its corresponding maximal torus as $S_{i}$. Then the coordinate flags would have the following splitting into line bundles with respect to each torus $S_{i}$:
\begin{align}
\mc{V}_a=\bigoplus_{\alpha,i}\mc{L}_{a,\alpha}^{(i)},\qquad\mc{L}_{a,\alpha}^{(i)}=\mc{O}(d_{a,\alpha}^{(i)})
\end{align}

In this case the corresponding formula can be written as: For $\mc{F}_i\in K_{T}(X_{i})$, we have:
\begin{equation}
\begin{aligned}
&\chi(X_1\times\cdots\times X_k,\text{Vertex}\otimes(\mc{F}_1\boxtimes\cdots\boxtimes\mc{F}_k))\\
=&\frac{1}{|W|}\sum_{\{d_{a,\alpha}^{(i)}\}}q^{-\frac{1}{2}\text{deg}\mc{T}^{1/2}_{vir,X_1\times\cdots\times X_k}}\prod_{a,\alpha,i}z_{a,i}^{d_{a,\alpha}^{(i)}}\int_{\gamma_{\chi_1}}\cdots\int_{\gamma_{\chi_k}}\frac{\mc{F}_1(s_1)\cdots\mc{F}_k(s_k)d_{Haar}s}{\wedge^*(T_{vir,moving}^{X_1\times\cdots\times X_k})}\\
=&\chi(X_1,\text{Vertex}\otimes\mc{F}_1)\cdots\chi(X_k,\text{Vertex}\otimes\mc{F}_k)
\end{aligned}
\end{equation}

Here the Weyl group $W=W_1\times\cdots\times W_k$ is the product of the Weyl group for each $G_{\mbf{v}_i}$.

Combining the result above we have a more elaborate result on the index limit:
\begin{prop}\label{Bare-proposition-facto}
Given $\tau(\mc{V})\in K_{T}(pt)[\cdots,x_{i1},\cdots,x_{in_i},\cdots]$, the vertex function with descedents $\text{Vertex}^{(\tau(\mc{V}))}_{M(\mbf{v},\mbf{w})}(z,a)$ satisfies the following factorisation property as the equivariant parametre $a\rightarrow0$:
\begin{align}\label{Bare-factorisation}
\lim_{a\rightarrow0}\text{Vertex}^{(\tau(\mc{V}))}_{M(\mbf{v},\mbf{w})}(z,a)=\sum_{\mbf{v}_1+\mbf{v}_2=\mbf{v}}q^{\Omega}\text{Vertex}^{(\tau(\mc{V}))}_{M(\mbf{v}_1,\mbf{w}_1)}(zq^{-\frac{C\mbf{v}_2-\Delta\mbf{w}_2}{2}})\otimes\text{Vertex}^{(1)}_{M(\mbf{v}_2,\mbf{w}_2)}(zq^{\frac{C\mbf{v}_1-\Delta\mbf{w}_1}{2}}p^{C\mbf{v}_1-(\Delta\mbf{w}_1)})
\end{align}
\end{prop}

\section{\textbf{Nonabelian stable envelopes}}
Here we review the construction of the nonabelian stable envelope constructed in \cite{AO17}.

Given an oriented framed quiver $Q$, let $T^*\text{Rep}_{Q}(\mbf{v},\mbf{w})$ be the quiver representation we introduced in \ref{quiver-representation}. Let 
\begin{align}
\mu:T^*\text{Rep}(\mbf{v},\mbf{w})\rightarrow\mf{g}_{\mbf{v}}^*,\qquad G=\prod_{i\in I}GL(V_i)
\end{align}
and we mainly focus on the zero locus $\mu_{\mbf{v},\mbf{w}}^{-1}(0)$, and by definition, the Nakajima quiver variety is defined as:
\begin{align}
M_{\bm{\theta},0}(\mbf{v},\mbf{w}):=[\mu^{-1}_{\mbf{v},\mbf{w}}(0)^{\bm{\theta}-ss}/G_{\mbf{v}}]
\end{align}

This is an open substack of the quotient stack
\begin{align}
M_{\bm{\theta},0}(\mbf{v},\mbf{w})\subset\mc{X}_{\mbf{v},\mbf{w}}=[\mu^{-1}_{\mbf{v},\mbf{w}}(0)/G_{\mbf{v}}]\subset\mc{R}_{\mbf{v},\mbf{w}}=[T^*\text{Rep}(\mbf{v},\mbf{w})/G_{\mbf{v}}]
\end{align}

Let $V_i'$ be collection of vector spaces of $\text{dim}(V_i')=v_i$, and denote
\begin{align}
G'_{\mbf{v}}=\prod_{i\in I}GL(V_i')\cong G
\end{align}
Define
\begin{align}
Y(\mbf{v},\mbf{w}+\mbf{v}):=\mu^{-1}(0)_{\mbf{v},\mbf{w}+\mbf{v},iso}/G_{\mbf{v}}'
\end{align}

Here the stability condition for $\mu^{-1}(0)_{\mbf{v},\mbf{w}+\mbf{v}}$ is chosen as $\bm{\theta}'=(\bm{\theta},1,\cdots,1)$.
The "iso" refers to the locus of points where the framing maps
\begin{align}
V_i'\rightarrow V_i,\text{respectively}, V_i\rightarrow V_i'
\end{align}
are isomorphisms. Clearly we have that $\mu^{-1}(0)_{\mbf{v},\mbf{w}+\mbf{v},iso}\subset\mu^{-1}(0)_{\mbf{v},\mbf{w}+\mbf{v},G_{\mbf{v}}-stable}$. By construction, there is a $G_{\mbf{v}}$-equivariant map
\begin{align}\label{restriction-nonabelian}
i:T^*\text{Rep}(\mbf{v},\mbf{w})\hookrightarrow Y_{\mbf{v},\mbf{w}+\mbf{v}}
\end{align}
whose supplements quiver maps by
\begin{align}
(\phi,-\phi^{-1}\circ\mu_{GL(V_i)})\in\text{Hom}(V_i',V_i)\oplus\text{Hom}(V_i,V_i')
\end{align}
for a framing isomorphism
\begin{align}
\phi:V_i'\rightarrow V_i
\end{align}

After doing the quotient over $G_{\mbf{v}}$, we have the induced map
\begin{align}
i:[T^*\text{Rep}(\mbf{v},\mbf{w})/G_{\mbf{v}}]\hookrightarrow[Y_{\mbf{v},\mbf{v}+\mbf{w}}/G_{\mbf{v}}]
\end{align}

We have the following pullback diagram

\begin{equation}
\begin{tikzcd}
\mc{X}_{\mbf{v},\mbf{w}}\arrow[r]\arrow[d]&\mc{R}_{\mbf{v},\mbf{w}}\arrow[d,"\mu_{G'}\circ i"]\\
0/G'_{\mbf{v}}\arrow[r]&\mf{g}_{\mbf{v}}'/G_{\mbf{v}}'
\end{tikzcd}
\end{equation}

Let $U\subset G_{\mbf{v}}'$ be the one-dimensional torus with its defining weight $u$ on each $V_i'$. We have
\begin{align}
M(\mbf{v},\mbf{w})\sqcup M(\mbf{v},\mbf{v})\subset M(\mbf{v},\mbf{w}+\mbf{v})^U
\end{align}

Since $U$ commutes with $T\times G'$, stable envelopes give a $K_{T\times G'}(X)$-linear map
\begin{align}\label{stable-nonabelian-stable}
\text{Stab}_U:K_{T\times G'}(M(\mbf{v},\mbf{w}))\rightarrow K_{T\times G'}(M(\mbf{v},\mbf{w}+\mbf{v}))
\end{align}

Here the stable envelope is chosen such that:
\begin{itemize}
	\item The slope $s'$ of the stable envelope is chosen such that $s=\epsilon\cdot(\text{ample line bundle})$ for $\epsilon>0$ infinitesimally close to $0$, which is compatible with the stable envelope we have chosen for $M(\mbf{v},\mbf{w})$.
	\item The polarisation for $M(\mbf{v},\mbf{w}+\mbf{v})$ is chosen such that:
	\begin{align}
	T^{1/2}_{M(\mbf{v},\mbf{w}+\mbf{v})}=T^{1/2}_{M(\mbf{v},\mbf{w})}+\sum_{i\in I}q^{-1}\text{Hom}(V_i,V_i')
	\end{align}
\end{itemize}

Since $\text{Stab}(\alpha)$ is $G'$-equivariant, it descends to a class on $Y_{\mbf{v},\mbf{v}+\mbf{w}}/G_{\mbf{v}}$.
Then combining with the map $i:[T^*\text{Rep}(\mbf{v},\mbf{w})/G_{\mbf{v}}]\hookrightarrow[Y_{\mbf{v},\mbf{v}+\mbf{w}}/G_{\mbf{v}}]=[M(\mbf{v},\mbf{v}+\mbf{w})_{iso}/G_{\mbf{v}}']$, we set:
\begin{align}
\mbf{s}_{\alpha}=i^*\text{Stab}_U(\alpha)\in K_{T}(T^*\text{Rep}(\mbf{v},\mbf{w})/G_{\mbf{v}}])
\end{align}

We define the factor:
\begin{align}
\Delta_{q}=\wedge^*(\bigoplus_{i\in I} q^{-1}\text{Hom}(V_i,V_i))=\wedge^*(q^{-1}\mf{g}_{\mbf{v}})
\end{align}

Thus we define
\begin{align}
g_{\alpha}=\Delta_{q}^{-1}\mbf{s}_{\alpha}\in K_{T}(T^*\text{Rep}(\mbf{v},\mbf{w})/G_{\mbf{v}}])
\end{align}

This is called the \textbf{nonabelian stable envelope} for the class $\alpha\in K_{T}(X)$. It is supported on the stack $\mc{X}_{\mbf{v},\mbf{w}}\subset\mc{R}_{\mbf{v},\mbf{w}}$.

The usage of the nonabelian stable envelope is to replace the relative insertion by the nonsingular insertion in the computation of the capping operators.

\section{\textbf{Capped and Capping operators}}
\subsection{Capped and capping operators}
\subsubsection{Capped vertex and large framing vanishing}
For the quasimap moduli space $\text{QM}^{\mbf{d}}_{\text{rel }p_2}(X)$, it has a natural evaluation map $\hat{\text{ev}}_{p_2}:\text{QM}^{\mbf{d}}_{\text{rel }p_2}(X)\rightarrow X$.
We define the capped descendents by:
\begin{align}
\text{Cap}^{\tau}(z):=\sum_{\mbf{d}\in C_{ample}}z^{\mbf{d}}\hat{\text{ev}}_{p_2,*}(\hat{\mc{O}}_{vir}^{\mbf{d}}\otimes\tau(\mc{V}|_{p_1}))\in K(X)\otimes\mbb{Q}[[z]]
\end{align}
It is easy to see that:
\begin{align}
\text{Cap}^{\tau}(z)=\tau(V)K_{X}^{1/2}+O(z)
\end{align}

Here we review the large framing vanishing for the capped vertex function with descendents. For details see section 7.5.6 of \cite{O15}.

Let us suppose that the vector bundle $\mc{V}_{i}$ over the source curve $C$ can be decomposed as:
\begin{align}
\mc{V}_i\cong\bigoplus_{k}\mc{O}(d_{ik})
\end{align}
by the weight with respect to the $\mbb{C}^*_p$-action over $\mbb{C}$. We can have the following decomposition:
\begin{align}
\mc{V}_{i,\mbf{n}}:=\bigoplus_{d_{ik}<0}\mc{O}(d_{ik}),\qquad\mc{V}_{i,\mbf{p}}\cong\bigoplus_{d_{ik}>0}\mc{O}(d_{ik})
\end{align}

we have the following identity:
\begin{align}
\text{deg}(\mc{V})=\text{deg}(\mc{V}_{\mbf{n}})+\text{deg}(\mc{V}_{\mbf{p}})\in C_{ample}^{\vee}
\end{align}

\begin{defn}
Fix the stability parametre $\bm{\theta}$ for the quiver varieties, we say that the polarisation $T^{1/2}$ is large with respect to $\bm{\theta}'\in C_{ample}$ if
\begin{align}
\text{deg}(\mc{T}_{\mbf{p}}^{1/2})>\bm{\theta}\cdot\text{deg}(f),\qquad-\text{deg}(\mc{T}_{\mbf{n}}^{1/2})\geq\bm{\theta}\cdot\text{deg}(f)
\end{align}
for any nonconstant stable quasimap $f$ with stability parametre $\bm{\theta}'$.
\end{defn}

Now we choose the polarisation that is suitable for our situation.

We take:
\begin{align}\label{polarisaion-large-framing}
T^{1/2}M(\mbf{v},\mbf{w})=(\text{The polarisation in \ref{polarisation-0}})+\mc{L}-q^{-1}\mc{L}^{-1}
\end{align}

Here $\mc{L}=\otimes_{i\in I}(\text{det}(V_i))^{-C\theta_i}$ with $C>>0$. We also choose $\mbf{w}$ such that $w_i=C'|\theta_i|+(\text{something small})$ with $C'>>C>>0$. 

The following theorem is the statement of the large framing vanishing:
\begin{thm}
Given $\tau\in K_{G_{\mbf{v}}}(pt)$ and $\bm{\theta}\in C_{ample}$ chosen in the Corollary 7.5.6 of \cite{O15}. If the polarisation $T^{1/2}$ is large with respect to $(\text{deg}(\tau))\bm{\theta}$, we have:
\begin{align}
\text{Cap}^{(\tau)}(z)=\tau(\mc{V})\mc{K}_{X}^{1/2}
\end{align}
\end{thm}

For the proof of the theorem see the Theorem 7.5.23 in \cite{O15}.

\subsubsection{Capping operators}
Another important ingredients is the capping operator for the GIT quotient. Consider the quasimap moduli space $\text{QM}^{\mbf{d}}_{\text{rel }p_1,\text{nonsing }p_2}(X)$. We first introduce the general setting for the capping operators.

We can define
\begin{align}\label{capping}
\Psi(z)=\sum_{\mbf{d}\in C_{ample}^{\vee}}z^{\mbf{d}}(\hat{\text{ev}}_{p_1}\times\text{ev}_{p_2})_*(\hat{\mc{O}}_{vir}^{\mbf{d}})\in K(X)^{\otimes2}_{loc}\otimes\mbb{Q}[[z]]
\end{align}

Here $\hat{\text{ev}}_{p_1}\times\text{ev}_{p_2}:\text{QM}^{\mbf{d}}_{\text{rel }p_1,\text{nonsing }p_2}(X)\rightarrow X\times X$ is the evaluation map, and $\text{ev}_{p_2}$ is only proper on the $\mbb{C}_{p}^*$-fixed points.

Moreover, the capping operator can be seen as an adjoint of the linear operator $\Psi^{\dagger}$ on $K_{T}(M(\mbf{v},\mbf{w}))\rightarrow K_{T}(M(\mbf{v},\mbf{w}))_{loc}[[z]]$ via:
\begin{align}
\Psi^{\dagger}(z)(\alpha)=\sum_{\mbf{d}}z^{\mbf{d}}\text{ev}_{p_2*}(\hat{\text{ev}}_{p_1}^*(\alpha)\otimes\hat{\mc{O}}_{vir}^{\mbf{d}}):K_{G}(X)\rightarrow K_{G\times\mbb{C}_{p}^*}(X)_{loc}[[z]]
\end{align}
and it is an invertible operator on $K_{T}(X)_{loc}$.

It is easy to see that with respect to the bilinear form $(-,-)$ in \ref{twisted-bilinear-form}, we have the following expansion:
\begin{align}
\Psi(z)=1+O(z)\in\text{End}(K_{T}(X)_{loc})[[z^{\mbf{d}}]]
\end{align}

It is known that \cite{AO17} the adjoint of the capping operator with the input $\alpha$ can be replaced by the descendent insertion with the normalised nonabelian stable envelopes of $\alpha$:
\begin{align}
\Psi^{\dagger}(z)(\alpha)=\text{Vertex}^{(g_{\alpha})}(z)=\sum_{\mbf{d}}z^{\mbf{d}}\text{ev}_{p_2*}(\text{ev}_{p_1}^*(g_{\alpha})\otimes\hat{\mc{O}}_{vir}^{\mbf{d}}):K_{G}(\mc{X})\rightarrow K_{G\times\mbb{C}_{p}^*}(\mc{X})_{loc}[[z]]
\end{align}

Thus the adjoint of the capping operator with the input by the class $\alpha$ is equivalent to the bare vertex function with descendents $g_{\alpha}$.

Now since $\pi$ is an isomorphism over $p_2$, we have:
\begin{align}
\text{det}(H^*(\mc{V}\otimes\pi^*(\mc{O}_{p_2})))=\text{ev}_2^*\mc{L},\qquad\mc{L}=\text{det}(\mc{V})
\end{align}

Using this definition, one can define the operator:
\begin{align}
\mbf{M}_{\mc{L}}(z)=\sum_{\mbf{d}\in C_{ample}}z^{\mbf{d}}(\hat{\text{ev}}_{p_1}\times\hat{\text{ev}}_{p_2})_*(\hat{\mc{O}}_{vir}^{\mbf{d}}\text{det}(H^*(\mc{V}\otimes\pi^*(\mc{O}_{p_1}))))\mbf{G}^{-1}
\end{align}

Here $\mbf{G}(z)=\sum_{\mbf{d}\in C_{ample}}z^{\mbf{d}}(\hat{\text{ev}}_{p_1}\times\hat{\text{ev}}_{p_2})_*(\hat{\mc{O}}_{vir}^{\mbf{d}}))$ is the gluing operator,.

The quantum difference operator is defined as an integral $K$-theory class. Moreover, by construction we have:
\begin{align}
\mbf{M}_{\mc{L}}(z)=\mc{L}+O(z)
\end{align}

Using the equivariant degeneration of the curve $C\mapsto C_1\cup C_2$, we have the following formula:
\begin{align}
\Psi(zp^{\mc{L}})\mc{L}=\mbf{M}_{\mc{L}}(z)\Psi(z)
\end{align}

This is a difference equation over the Kahler variables. Moreover one can also have the difference equation over the equivariant variables.

We define the operator:
\begin{align}
\Psi^{\sigma}:K_T(X)\rightarrow K_{T\times\mbb{C}^*_{p}}(X)_{loc}\otimes\mbb{Q}[[z]]
\end{align}
as the capping operator but for twisted quasimaps, i.e. the action of $\mbb{C}_{p}^*$ on $X$ is given by $\sigma$.

By computation one can show that:
\begin{align}
\Psi^{\sigma}=J(aq^{\sigma})E(x,\sigma)
\end{align}
Here:
\begin{align}
E(x,\sigma)=\frac{\text{edge term for }f_{\sigma}}{(\text{edge form for }f_0)|_{a\mapsto q^{\sigma}a}}
\end{align}
Here $f_{\sigma}$ is defined such that for each fixed point $x\in X^{\sigma}$ we can associate a twisted quasimap $f$ such that $f(c)=x$ for all $c\in C$. If the variety $X$ admits the polarisation  $T=T^{1/2}+q^{-1}(T^{1/2})^{\vee}$, we have the following formula:
\begin{align}
E(x,\sigma)=(z_{\text{shifted}})^{\langle\mu(x),-\otimes\sigma\rangle}\Phi((1-pq)(T_{x,\sigma}^{1/2}-T_{x,0}^{1/2}))
\end{align}
and in our case we will denote $E(x,\sigma)$ by $E_{\sigma}(z,a)$

We can define the shift operator as:
\begin{align}
\mbf{S}_{\sigma}=\sum_{\mbf{d}}z^{\mbf{d}}\hat{\text{ev}}_{*}(\text{twisted quasimaps},\hat{\mc{O}}_{vir}^{\mbf{d}})\mbf{G}^{-1}
\end{align}

For the detail of the above definition see \cite{O15}.
One immediately has the following difference equations:
\begin{align}
\Psi(z,q^{\sigma}a)E_\sigma(z,a)=\mbf{S}_{\sigma}(z,a)\Psi(z,a)
\end{align}

The conjugation of the shift operator to qKZ is given by the stable envelopes:
\begin{align}\label{R-matrix-shift}
\text{Stab}_{s,\sigma}^{-1}\tau_{\sigma}^{-1}\mbf{S}_{\sigma}\text{Stab}_{s,\sigma}=\tau_{\sigma}^{-1}(-1)^{\frac{\text{codim}}{2}}z^{deg}\mc{R}^s_{\sigma}
\end{align}
with $s\in-C_{ample}$. Here $\tau_{\sigma}$ is the shifting $a\mapsto q^{\sigma}a$.

\subsection{\textbf{Quantum difference equations and cocycle identities}}
It is known that the capping operator $\Psi(z)$ satisfies the quantum difference equation:
\begin{align}
\Psi(zp^{\mc{L}})\mc{L}=\mbf{M}_{\mc{L}}(z)\Psi(z)
\end{align}

It has been proved that \cite{OS22} the quantum difference operator $\mbf{M}_{\mc{L}}(z)$ can be expressed in terms of the generators of the MO quantum affine algebras $U_{q}(\hat{\mf{g}}_{Q})$ up to a constant operator:
\begin{align}
\text{Stab}_{+,T^{1/2},s}^{-1}\mbf{M}_{\mc{L}}(z)\text{Stab}_{+,T^{1/2},s}=\text{Const}\cdot\Delta_{s}(\mbf{B}_{\mc{L}}^s(u,z))
\end{align}
with $s\in-C_{ample}$.

Here $\text{Const}$ is some function in the variables $q$, $t_e$ and the Kahler variables $z$. Also:
\begin{align}
\mbf{B}_{\mc{L}}^s(u,z)=\mc{L}\prod^{\leftarrow}_{w\in[s,s-\mc{L})}\mbf{B}_{w}(u,z)
\end{align}

Thus for the fundamental solution to the difference equation:
\begin{align}
\Psi_{s}(zp^{\mc{L}})\mc{L}=\mbf{B}_{\mc{L}}^s(u,z)\Psi_s(z)
\end{align}

It is connected to the capping operator $\Psi(z)$ via:
\begin{equation}
\begin{aligned}
&\mbf{M}_{\mc{L}}(z)\Psi(z)=\mbf{M}_{\mc{L}}(z)\text{Stab}_{+,T^{1/2},s}\text{Stab}_{+,T^{1/2},s}^{-1}\Psi(z)\\
=&\text{Stab}_{+,T^{1/2},s}\text{Stab}_{+,T^{1/2},s}^{-1}\mbf{M}_{\mc{L}}(z)\text{Stab}_{+,T^{1/2},s}\text{Stab}_{+,T^{1/2},s}^{-1}\Psi(z)\\
=&\text{Const}\cdot\text{Stab}_{+,T^{1/2},s}\Delta_{s}(\mbf{B}^s_{\mc{L}}(u,z))\text{Stab}_{+,T^{1/2},s}^{-1}\Psi(z)=\Psi(zp^{\mc{L}})\mc{L}
\end{aligned}
\end{equation}

Thus we have that:
\begin{align}
\text{Const}\cdot\mbf{B}^{s}_{\mc{L}}(u,z)\Psi(z)=\Psi(zp^{\mc{L}})\mc{L}
\end{align}

\subsection{Smaller capping operators}

We define a smaller capping operator:
\begin{align}
\Psi_{X^A}(z)=\sum_{\mbf{d}\in C_{ample}^{\vee}}z^{\mbf{d}}(\hat{\text{ev}}_{p_1}\times\text{ev}_{p_2})_*(\hat{\mc{O}}_{vir}^{\mbf{d}})\in K(X^A)_{loc}^{\otimes 2}\otimes\mbb{Q}[[z]]
\end{align}
which is defined over the quasimap moduli space $QM_{\text{rel }p_1,p_2}(X^A)$ over the fixed point component $X^A$, it is easy to see that if we write $X^A=\sqcup_{\alpha}F_{\alpha}$, $\Psi_{X^A}(z)$ is a block diagonal operator sending $K_{T}(F_{\alpha})_{loc}$ to $K_{T}(F_{\alpha})[[z]]$.

It is not hard to check that:
\begin{align}
\text{Cap}^{(\tau)}_{X^A}(z)=\Psi_{X^A}(z)\text{Vertex}_{X^A}^{(\tau)}(z)
\end{align}

If we fix $X=M(\mbf{v},\mbf{w})$ and $A$ such that $X^A=\sqcup_{\mbf{v}_1+\mbf{v}_2=\mbf{v}}M(\mbf{v}_1,\mbf{w}_1)\times M(\mbf{v}_2,\mbf{w}_2)$. We denote the correpsonding smaller capping operator as $\Psi^{(\mbf{w}_1),(\mbf{w}_2)}(z)$. Note that since $\Psi^{(\mbf{w}_1),(\mbf{w}_2)}(z)$ has its Kahler variable from $M(\mbf{v}_1,\mbf{w}_1)\times M(\mbf{v}_2,\mbf{w}_2)$, we usually denote it by $\Psi^{(\mbf{w}_1),(\mbf{w}_2)}(z_1,z_2)$.

We define the operator:
\begin{align}
\mbf{J}^{(\mbf{w}_1),(\mbf{w}_2)}(z)=i^*\Psi^{(\mbf{w}_1+\mbf{w}_2)}(z)(i^*)^{-1}(\Psi^{(\mbf{w}_1),(\mbf{w}_2)}(zq^{\frac{\Delta\mbf{w}_2-C\mbf{v}_2}{2}}, zq^{\frac{C\mbf{v}_1-\Delta\mbf{w}_1}{2}}))^{-1}
\end{align}

Then choose a splitting $\mbf{w}=\mbf{w}_1+a\mbf{w}_2$ and define the operator:
\begin{align}
Y^{(\mbf{w}_1),(\mbf{w}_2)}(z)=\lim_{a\rightarrow0}\mbf{J}^{(\mbf{w}_1),(\mbf{w}_2)}(z)
\end{align}

It is easy to see that $Y^{(\mbf{w}_1),(\mbf{w}_2)}(z)$ is well-defined and it is a lower-triangular operator satisfying the dynamical cocycle identity:
\begin{equation}
Y^{(\mbf{w}_1+\mbf{w}_2),(\mbf{w}_3)}(z)Y^{(\mbf{w}_1),(\mbf{w}_2)}(zq^{\frac{\Delta\mbf{w}_3-C\mbf{v}_3}{2}})=Y^{(\mbf{w}_1),(\mbf{w}_2+\mbf{w}_3)}(z)Y^{(\mbf{w}_2),(\mbf{w}_3)}(zq^{-\frac{\Delta\mbf{w}_1-C\mbf{v}_1}{2}})
\end{equation}

\subsection{qKZ equation to wKZ equation}
It has been shown that the capping operator $\Psi(z,a)$ satisfies the shift operator difference equation:
\begin{align}
\Psi(z,ap)E(z,a)=S(z,a)\Psi(z,a)
\end{align}

Also we know that:
\begin{equation}
\begin{aligned}
&\text{Stab}_{+,T^{1/2},s}^{-1}S(z,a)\text{Stab}_{+,T^{1/2},s}=z_{(1)}\mc{R}^s(a)\\
&\text{Stab}_{+,T^{1/2},\infty}^{-1}E(z,a)\text{Stab}_{+,T^{1/2},\infty}=z_{(1)}\mc{R}^{\infty}(a)
\end{aligned}
\end{equation}

Thus we found that $\Phi_s(z,a):=\text{Stab}_{+,T^{1/2},s}^{-1}\Psi(z,a)\text{Stab}_{+,T^{1/2},\infty}$ solves the 2-point qKZ equation:
\begin{align}
\Phi_s(z,ap)z_{(1)}\mc{R}^{\infty}(a)=z_{(1)}\mc{R}^s(a)\Phi_s(z,a)
\end{align}
for the slope $s\in-C_{ample}$ which is really close to $0$.

We mainly study the limit behavior of the qKZ equation when $a\rightarrow0$.

\begin{prop}\label{wKZ-proposition}
As $a\rightarrow0$, the qKZ equation goes to the following wKZ equation:
\begin{align}\label{wKZ-equation-01}
z_{(1)}^{\lambda}(R_0^{-,(\mbf{w}_1),(\mbf{w}_2)})^{-1}q^{\Omega}\Phi_s(z,0)=\Phi_s(z,0)q^{\Omega}z_{(1)}^{\lambda}
\end{align}
\end{prop}
\begin{proof}
Now we take $a\rightarrow0$, and in this case using the KT factorisation for the geometric $R$-matrix \ref{factorisation-geometry}:
\begin{align}
\mc{R}^{s}(u)=\prod_{i<s}^{\leftarrow}R_{w_i}^{-}R_{\infty}\prod_{i\geq s}^{\leftarrow}R_{w_i}^{+}
\end{align}

One can see that under our choice of the slope $s$, the wall $R$-matrix $R_{w}^{-}$ appearing in the first infinite product has the matrix coefficients as $a^{-\langle\mc{L}_{w},\alpha\rangle}$, and here the wall $w$ is defined as $\langle s,\alpha_{w}\rangle=n\leq0$. Thus we can see that as $a\rightarrow0$, the one that does not contain zero goes to the identity operator. Similar argument for $\prod_{i\geq0}^{\leftarrow}R_{w_i}^{+}$ shows that as $a\rightarrow0$, it will go to the identity operator.

Thus $\Phi_s(z,0)$ has the equaton written as:
\begin{align}
z_{(1)}^{\lambda}(R_0^{-,(\mbf{w}_1),(\mbf{w}_2)})^{-1}q^{\Omega}\Phi_s(z,0)=\Phi_s(z,0)q^{\Omega}z_{(1)}^{\lambda}
\end{align}

Here $\mc{R}_{0}^{-}=\prod^{\leftarrow}_{w}R_{w}^-$ is the ordered product of the wall $R$-matrix $R_{w}^-$ such that the wall $w$ corresponds to the linear equation of the form $(\alpha,s)=0$.

\end{proof}

Now we evaluate $a$ to be $0$, we will have the following formula:
\begin{align}
i^*\Psi^{(\mbf{w}_1+\mbf{w}_2)}(z,0)(i^*)^{-1}=Y^{(\mbf{w}_1),(\mbf{w}_2)}(z)q^{\Omega}\Psi^{(\mbf{w}_1),(\mbf{w}_2)}(zq^{\frac{\Delta\mbf{w}_2-C\mbf{v}_2}{2}}, zq^{\frac{C\mbf{v}_1-\Delta\mbf{w}_1}{2}}))
\end{align}

We denote:
\begin{align}\label{fusion-operator}
H^{(\mbf{w}_1),(\mbf{w}_2)}(z)=\lim_{a\rightarrow0}\text{Stab}_{+,T^{1/2},s}^{-1}(i^*)^{-1}J^{(\mbf{w}_1),(\mbf{w}_2)}(z)i^*\text{Stab}_{+,T^{1/2},\infty}
\end{align}
and we call this operator as the \textbf{fusion operator}.

Now since $q^{\Omega}z_{(1)}^{\lambda}$ commutes with any block-diagonal operator. From the definition of $Y^{(\mbf{w}_1),(\mbf{w}_2)}(z)$ we can see that the fusion operator solves the following \textbf{wKZ equation}:
\begin{align}
z_{(1)}^{\lambda}(R_{0}^{-,(\mbf{w}_1),(\mbf{w}_2)})^{-1}q^{\Omega}H^{(\mbf{w}_1),(\mbf{w}_2)}(z)=H^{(\mbf{w}_1),(\mbf{w}_2)}(z)q^{\Omega}z_{(1)}^d
\end{align}

Here we prove the rationality of the fusion operator $H^{(\mbf{w}_1),(\mbf{w}_2)}(z)$, in this subsection we prove the rationality of the non-diagonal part.

\begin{lem}
For a given block diagonal operator $D^{(\mbf{w}_1),(\mbf{w}_2)}$ there exists unique lower triangular solution $J^{(\mbf{w}_1),(\mbf{w}_2)}(z)$ of the wKZ equation:
\begin{align}
z_{(1)}^{\lambda}(R_{0}^{-,(\mbf{w}_1),(\mbf{w}_2)})^{-1}q^{\Omega}J^{(\mbf{w}_1),(\mbf{w}_2)}(z)=J^{(\mbf{w}_1),(\mbf{w}_2)}(z)q^{\Omega}z_{(1)}^d
\end{align}
such that the diagonal part of $J^{(\mbf{w}_1),(\mbf{w}_2)}(z)$ is given by $D^{(\mbf{w}_1),(\mbf{w}_2)}(z)$.
\end{lem}
\begin{proof}
We denote $\tilde{R}:=q^{-\Omega}(R_{0}^{-,(\mbf{w}_1),(\mbf{w}_2)})^{-1}q^{\Omega}$. We can see that $\tilde{R}$ is strictly lower triangular, and we have the wKZ equation of the form:
\begin{align}\label{wKZ-equation}
\text{Ad}_{q^{-\Omega}z_{(1)}^{\lambda}}(J^{(\mbf{w}_1),(\mbf{w}_2)}(z))=\tilde{R}J^{(\mbf{w}_1),(\mbf{w}_2)}(z)
\end{align}

Now if we take the $\bm{\alpha}$-th component of the equation we obtain:
\begin{align}
\text{Ad}_{q^{-\Omega}z_{(1)}^{\lambda}}(J^{(\mbf{w}_1),(\mbf{w}_2)}_{(\bm{\alpha})}(z))=J^{(\mbf{w}_1),(\mbf{w}_2)}_{(\bm{\alpha})}(z)+\cdots
\end{align}
such that the dots stand for terms $J^{(\mbf{w}_1),(\mbf{w}_2)}_{(\bm{\beta})}(z)$ with $\beta<\alpha$. By induction, all the term can be expressed via the diagonal term $D^{(\mbf{w}_1),(\mbf{w}_2)}(z)$.

Also note that if $D^{(\mbf{w}_1),(\mbf{w}_2)}(z)=1$, we denote the corresponding solution as $E^{(\mbf{w}_1),(\mbf{w}_2)}(z)$. It is easy to check that $J^{(\mbf{w}_1),(\mbf{w}_2)}_{(\bm{\alpha})}(z)=E^{(\mbf{w}_1),(\mbf{w}_2)}_{(\bm{\alpha})}(z)D^{(\mbf{w}_1),(\mbf{w}_2)}_{(\bm{\alpha})}(z)$
\end{proof}

\begin{prop}\label{rationality-for-wKZ}
The matrix coefficients of $E^{(\mbf{w}_1),(\mbf{w}_2)}_{(\bm{\alpha})}(z)$ is rational.
\end{prop}
\begin{proof}
By the equation \ref{wKZ-equation} we have that
\begin{align}
\text{Ad}_{q^{-\Omega}z_{(1)}^{\lambda}}(E^{(\mbf{w}_1),(\mbf{w}_2)}(z))=\tilde{R}E^{(\mbf{w}_1),(\mbf{w}_2)}(z)
\end{align}

We write down $\tilde{R}=1+\sum_{\theta\cdot\alpha>0}R_{\alpha}$, $E^{(\mbf{w}_1),(\mbf{w}_2)}(z)=\sum_{\theta\cdot\alpha>0}E^{(\mbf{w}_1),(\mbf{w}_2)}_{\alpha}(z)$, and now we have that:
\begin{align}
\text{Ad}_{q^{-\Omega}z_{(1)}^{\lambda}}(E^{(\mbf{w}_1),(\mbf{w}_2)}_{\alpha}(z))-E^{(\mbf{w}_1),(\mbf{w}_2)}_{\alpha}(z)=\sum_{\substack{\alpha_1+\alpha_2=\alpha\\\alpha_2\neq\alpha}}\tilde{R}_{\alpha_1}E^{(\mbf{w}_1),(\mbf{w}_2)}_{\alpha_2}(z)
\end{align}
which implies the following recursive formula:
\begin{align}
E^{(\mbf{w}_1),(\mbf{w}_2)}_{\alpha}(z)=\frac{1}{\text{Ad}_{q^{-\Omega}z_{(1)}^{\lambda}}-1}\sum_{\substack{\alpha_1+\alpha_2=\alpha\\\alpha_2\neq\alpha}}\tilde{R}_{\alpha_1}E^{(\mbf{w}_1),(\mbf{w}_2)}_{\alpha_2}(z)
\end{align}

It remains to prove that $R_{\alpha}$ has rational coefficients. Since the matrix coefficients of each wall $R$-matrix $R_{w}^{\pm}$ is monomial in $a$ and rational in $t_e,q$. It remains to show that when restricting to $\bigoplus_{\mbf{v}_1+\mbf{v}_2=\mbf{v}}K_{T}(M(\mbf{v}_1,\mbf{w}_1))\otimes K_{T}(M(\mbf{v}_2,\mbf{w}_2))$ with fixed $\mbf{v}$ and $\mbf{w}_1+\mbf{w}_2$, $R_{0}^{-}=\prod^{\leftarrow}_{0\in w}R_{w}^{-}$  contains only finitely many terms in the product.

This can be seen from the following fact: Since the wall $w$ is determined by the root $(\alpha,n)\in\mbb{N}^I\times\mbb{Z}$ such that $w$ corresponds to the hyperplane $(s,\alpha)+n=0$. For the wall $w$ in $R_{0}^-$, it corresponds to the wall $(s,\alpha)=0$, and the nontrivial elements in $R_{w}^{-}$ corresponding to the vector $\alpha\in\mbb{N}^I$ satisfies that it sends $K_{T}(M(\mbf{v}_1,\mbf{w}_1))\otimes K_{T}(M(\mbf{v}_2,\mbf{w}_2))$ to $K_{T}(M(\mbf{v}_1+k\alpha,\mbf{w}_1))\otimes K_{T}(M(\mbf{v}_2-k\alpha,\mbf{w}_2))$ for $k>0$. There exists maximal $\alpha$ such that $\alpha\leq\mbf{v}$. Thus we can see that only those wall $R$-matrices $R_{w}^{-}$ satisfying that the corresponding root $(\alpha,0)$ is no bigger than $\mbf{v}$ can survive. This proves the finiteness of the terms in the product. Thus we conclude the proposition.
\end{proof}

In the next subsection we are going to prove that 
for the fusion operator $H^{(\mbf{w}_1),(\mbf{w}_2)}(z)$ its corresponding diagonal part $D^{(\mbf{w}_1),(\mbf{w}_2)}(z)$ is equal to one.

\subsection{Diagonal part of $H^{(\mbf{w}_1),(\mbf{w}_2)}(z)$}

\begin{prop}\label{digaonal-argument}
The diagonal part $D^{(\mbf{w}_1),(\mbf{w}_2)}(z)$ of the operator $H^{(\mbf{w}_1),(\mbf{w}_2)}(z)$ is the identity operator.
\end{prop}

\begin{proof}
The proof is based on the virtual localisation for the capping operator, namely, we can consider the capping operator in the following form:
\begin{align}
\sum_{\mbf{d}}z^{\mbf{d}}\hat{\text{ev}}_{*}(\hat{\mc{O}_{vir}^{\mbf{d}}}\frac{1}{1-p^{-1}\psi_{p_1}})
\end{align}

Here the evaluation map is on the quasimap moduli space $QM_{\text{rel }p_2,p_1}^{\sim}$ of stable quasimaps with whose domain $C$ is a chain of rational curves joining $p_1$ and $p_2$. The connection with the definition in \ref{capping} is given by the $\mbb{C}_p^*$-localisation of the above operator.

Now we further do the $A$-localisation, namely we consider a quasimap $f\in(QM_{\text{rel }p_1',p_1}^{\sim}(X))^A$ and let $C'$ be the domain of the curve $f$. Since $f$ is fixed, there exists 
\begin{align}
\gamma:A\rightarrow\text{Aut}(C',p_2',p_1')
\end{align}
such that $af=f\gamma(a)$ for $a\in A$. In this case the virtual normal bundle is written as:
\begin{align}
N_{vir}=N_{vir}^{\text{fixed domain}}+\bigoplus_{\text{broken nodes p}}T_pC_i'\otimes T_pC_{i+1}'
\end{align}

Now in the case that we want to restrict to the block diagonal operator part of $\Psi$ from $K_{T}(F_{\alpha})$ to $K_{T}(F_{\alpha})$. If $f\in QM(F_{\alpha})$, the corresponding virtual normal bundle $N_{vir}$ has no broken nodes part. If the virtual normal bundle has no broken nodes part, this means that the domain curve $C'$ can be smoothened to such that $C'\cong\mbb{P}^1$ and so $p_1$ and $p_2$ are all in just one component. In this case we can see that $\gamma$ can only be trivial in this case, and thus the corresponding quasimap $f\in QM(F_{\alpha})$. Thus the argument has been reduced to the one of the index limit for the bare vertex functions. 

So for now the corresponding block diagonal part, the virtual normal bundle is equal to:
\begin{align}
N_{vir}=N_{vir}^{\text{fixed domain}}=T^{1/2}_{f(p_1),moving}+T^{1/2}_{f(p_2),moving}+\cdots
\end{align}
The rest of the computation is the same as the one in the index limit for the vertex function. For now we have that:
\begin{align}
i^*\Psi_{X}(z)((i^*)^{-1}(\alpha\otimes\beta))=q^{\Omega}\Psi_{X^A}(\alpha\otimes\beta)+\cdots\text{When restricted to the diagonal part}
\end{align}

Now we take $a\rightarrow0$ and:
\begin{equation}
\begin{aligned}
&\lim_{a\rightarrow0}\text{Stab}_s^{-1}\Psi^{(\mbf{w}_1+\mbf{w}_2)}(z,a)\text{Stab}_{\infty}=(i^*\text{Stab}_s)^{-1}i^*\Psi^{(\mbf{w}_1+\mbf{w}_2)}(z,0)(i^*)^{-1}(i^*\text{Stab}_{\infty})\\
=&H^{(\mbf{w}_1),(\mbf{w}_2)}(z)q^{\Omega}(i^*\text{Stab}_{\infty})^{-1}\Psi^{(\mbf{w}_1),(\mbf{w}_2)}(zq^{\frac{\Delta\mbf{w}_2-C\mbf{v}_2}{2}}, zq^{\frac{C\mbf{v}_1-\Delta\mbf{w}_1}{2}})(i^*\text{Stab}_{\infty})
\end{aligned}
\end{equation}

Now we use the fact that the operator
\begin{align}
\lim_{a\rightarrow0}a^{\langle\text{det}(T^{1/2}_{>0}),\sigma\rangle}i^*\text{Stab}_{+,T^{1/2},s}=\lim_{a\rightarrow0}a^{\langle\text{det}(T^{1/2}_{>0}),\sigma\rangle}i^*\text{Stab}_{+,T^{1/2},\infty}
\end{align}
is a block diagonal operator, for details of the result one can refer to the exercise $10.2.14$ in \cite{O15}. Thus via computation we further have that:
\begin{align}
i^*\Psi^{(\mbf{w}_1+\mbf{w}_2)}(z,0)(i^*)^{-1}=\mc{L}H^{(\mbf{w}_1),(\mbf{w}_2)}(z)q^{\Omega}\mc{L}^{-1}\Psi^{(\mbf{w}_1),(\mbf{w}_2)}(zq^{\frac{\Delta\mbf{w}_2-C\mbf{v}_2}{2}}, zq^{\frac{C\mbf{v}_1-\Delta\mbf{w}_1}{2}})
\end{align}
Here $\mc{L}$ is some line bundle that comes from the index limit $a\rightarrow0$ of $i^*\text{Stab}_{s}$.
Restricted to the diagonal part and we have that:
\begin{align}
\Psi^{(\mbf{w}_1),(\mbf{w}_2)}(zq^{\frac{\Delta\mbf{w}_2-C\mbf{v}_2}{2}}, zq^{\frac{C\mbf{v}_1-\Delta\mbf{w}_1}{2}})=D^{(\mbf{w}_1),(\mbf{w}_2)}(z)\Psi^{(\mbf{w}_1),(\mbf{w}_2)}(zq^{\frac{\Delta\mbf{w}_2-C\mbf{v}_2}{2}}, zq^{\frac{C\mbf{v}_1-\Delta\mbf{w}_1}{2}})
\end{align}

Thus the proof is finished.
\end{proof}

\subsection{Adjoint operator $H^{(\mbf{w}_1),(\mbf{w}_2),\dagger}(z)$}

First we have the following observation: The capping operator $\Psi:K_{T}(X)\rightarrow K_{T}(X)_{loc}[[z]]$ is given by a $K$-theory class in $K_{T}(X\times X)_{loc}[[z]]$. In the similar manner, the operator $\text{Stab}_{s}^{-1}\circ\Psi\circ\text{Stab}_{\infty}$ is represented by a $K$-theory class in $K_{T}(X^A\times X^A)_{loc}[[z]]$ via the convolution composition of the $K$-theory classes. Thus if we choose the block-diagonal part of the operator $\text{Stab}_{s}^{-1}\circ\Psi\circ\text{Stab}_{\infty}$, it is represented by a $K$-theory class in $\bigoplus_{\alpha}K_{T}(F_{\alpha}\times F_{\alpha})_{loc}[[z]]$.

It is convenient for us to first write down adjoint of the capping operator after doing the conjugation:
\begin{equation}\label{reduction-of-adjoint-capping}
\begin{aligned}
&\lim_{a\rightarrow0}\text{Stab}_{s}^{-1}\Psi^{\dagger}\text{Stab}_{\infty}(\alpha\otimes\beta)=\lim_{a\rightarrow0}\text{Stab}_s^{-1}\text{Vertex}_{X}^{g(\text{Stab}_{\infty}(\alpha\otimes\beta))}\\
=&\lim_{a\rightarrow0}(i^*\circ\text{Stab}_s)^{-1}[\underbrace{(\cdots)}_{\text{Coefficients as in \ref{index-limit}}}\text{Vertex}_{X^A}^{(i^*g(\text{Stab}_{\infty}(\alpha\otimes\beta))}]
\end{aligned}
\end{equation}

The second line comes from the virtual localisation, which is also used in the argument of the index limit of the vertex function. This implies that we can transfer the computation to the vertex function with descendents $i^*g(\text{Stab}_{\infty}(\alpha\otimes\beta)$.

For now we need to compute the block-diagonal part of the operator $\lim_{a\rightarrow0}\text{Stab}_{s}^{-1}\Psi^{\dagger}\text{Stab}_{\infty}$. By the above computation, it remains to compute the following limit:
\begin{align}
\lim_{a\rightarrow0}\text{Stab}_{s}^{-1}g(\text{Stab}_{\infty}(\alpha\otimes\beta))
\end{align}

But also note that up to some scaling of $a$, the operator
\begin{align}
\lim_{a\rightarrow0}a^{\langle\text{det}(T^{1/2}_{>0}),\sigma\rangle}\text{Stab}_{+,T^{1/2},s}=\lim_{a\rightarrow0}a^{\langle\text{det}(T^{1/2}_{>0}),\sigma\rangle}\text{Stab}_{+,T^{1/2},\infty}
\end{align}
is a block diagonal operator, for details of the result one can refer to the exercise $10.2.14$ in \cite{O15}. Now we have that:
\begin{equation}
\begin{aligned}
&\lim_{a\rightarrow0}\text{Vertex}_{X}(\text{Stab}_{s}^{-1}g(\text{Stab}_{\infty}(\alpha\otimes\beta)))(z)=\lim_{a\rightarrow0}\text{Stab}_{s}^{-1}\text{Vertex}_{X}(g(\text{Stab}_{s}(\alpha\otimes\beta)))(z)
\end{aligned}
\end{equation}

Now consider the following commutative diagram:
\begin{equation}
\begin{tikzcd}
K_{T\times G'}(M(\mbf{v},\mbf{w}))\arrow[r,"\text{Stab}_{U}"]&K_{T\times G'}(M(\mbf{v},\mbf{w}+\mbf{v}))\\
\oplus_{\mbf{v}_1+\mbf{v}_2=\mbf{v}}K_{T\times G'}(M(\mbf{v}_1,\mbf{w}_1)\times M(\mbf{v}_2,\mbf{w}_2))\arrow[r,"\text{Stab}_{U_1}\otimes\text{Stab}_{U_2}"]\arrow[u,"\text{Stab}_{s}"]&\oplus_{\mbf{v}_1+\mbf{v}_2=\mbf{v}}K_{T}(M(\mbf{v}_1,\mbf{v}_1+\mbf{w}_1)\times M(\mbf{v}_2,\mbf{v}_2+\mbf{w}_2))\arrow[u,"\text{Stab}_{s,U}"]
\end{tikzcd}
\end{equation}

Here the stable envelope $\text{Stab}_{U}$ is the stable envelope used in \ref{stable-nonabelian-stable}. 
The commutativity of the diagram comes from the uniqueness of the stable envelopes.

Now we use the notation $\text{Stab}_{\infty}=\frac{1}{\wedge^*(N_{X^A}^-)}i^*$ such that all of these classes lives in $K_{T}(\mc{R}_{\mbf{v}})_{loc}$. So now we have that

\begin{equation}
\begin{aligned}
&\lim_{a\rightarrow0}\text{Stab}_{\infty}^{-1}g(\text{Stab}_{\infty}(\alpha\otimes\beta))=\lim_{a\rightarrow0}\text{Stab}_{s}^{-1}g(\text{Stab}_{s}(\alpha\otimes\beta))\\
=&\lim_{a\rightarrow0}\text{Stab}_{s}^{-1}\frac{1}{\wedge^*(q^{-1}\mf{g}_{\mbf{v}})}j^*\text{Stab}_{U}\circ\text{Stab}_s(\alpha\otimes\beta)\\
=&\lim_{a\rightarrow0}\text{Stab}_s^{-1}\frac{1}{\wedge^*(q^{-1}\mf{g}_{\mbf{v}})}j^*\text{Stab}_{s,U}\circ(\text{Stab}_{U_1}(\alpha)\otimes\text{Stab}_{U_2}(\beta))
\end{aligned}
\end{equation}
Now if we restrict the formula to the block diagonal, we have that the block diagonal part should be equal to:
\begin{equation}
\begin{aligned}
&\lim_{a\rightarrow0}\text{Stab}_{\infty}^{-1}\frac{1}{\wedge^*(q^{-1}\mf{g}_{\mbf{v}}[1])\wedge^*(q^{-1}\mf{g}_{\mbf{v}}[-1])}\frac{1}{\wedge^*(q^{-1}\mf{g}_{\mbf{v}_1})\wedge^*(q^{-1}\mf{g}_{\mbf{v}_2})}j^*\text{Stab}_{\infty,U}\circ(\text{Stab}_{U_1}(\alpha)\otimes\text{Stab}_{U_2}(\beta))\\
=&\lim_{a\rightarrow0}\frac{1}{\wedge^*(N_{X^A}^-)}i^*\frac{1}{\wedge^*(q^{-1}\mf{g}_{\mbf{v}}[1])\wedge^*(q^{-1}\mf{g}_{\mbf{v}}[-1])}\frac{1}{\wedge^*(q^{-1}\mf{g}_{\mbf{v}_1})\wedge^*(q^{-1}\mf{g}_{\mbf{v}_2})}\times\\
&\times j^*(i^*_U)^{-1}(\wedge^*(N_{X'^A}^-)\text{Stab}_{U_1}(\alpha)\otimes\text{Stab}_{U_2}(\beta))
\end{aligned}
\end{equation}

Here we use the fact that $\lim_{a\rightarrow0}a^{\langle\text{det}(T^{1/2}_{>0}),\sigma\rangle}\text{Stab}_{+,T^{1/2},s}$ exists and it is block-triangular, and 
\begin{equation}
\begin{aligned}
&\mc{N}_{M(\mbf{v}_1,\mbf{v}_1+\mbf{w}_1)\times M(\mbf{v}_2,\mbf{v}_2+\mbf{w}_2)}^-=\mc{N}_{M(\mbf{v}_1,\mbf{w}_1)\times M(\mbf{v}_2,\mbf{w}_2)}^-+\sum_{i\in I}(\frac{\mc{V}_i''}{(q\mc{V}_i')'}+\frac{(\mc{V}_i')''}{\mc{V}_i'})\\
&\mf{g}_{\mbf{v}}[1]=\bigoplus_{i\in I}q^{-1}\text{Hom}(\mc{V}_i',\mc{V}_i''),\qquad\mf{g}_{\mbf{v}}[-1]=\bigoplus_{i\in I}q^{-1}\text{Hom}(\mc{V}_i'',\mc{V}_i')
\end{aligned}
\end{equation}

For the contribution $\wedge^*(q^{-1}\mf{g}_{\mbf{v}}[1])$, it goes to $1$ when $a\rightarrow0$.
Note that the map $j^*$ in \ref{restriction-nonabelian} identifies $\mc{V}'$ with $\mc{V}$, and now we have that:
\begin{equation}
\lim_{a\rightarrow0}j^*\wedge^*\mc{N}_{M(\mbf{v}_1,\mbf{v}_1+\mbf{w}_1)\times M(\mbf{v}_2,\mbf{v}_2+\mbf{w}_2)}^-=\lim_{a\rightarrow0}\wedge^*\mc{N}_{M(\mbf{v}_1,\mbf{w}_1)\times M(\mbf{v}_2,\mbf{w}_2)}^-\wedge^*(\mf{g}_{\mbf{v}}[1])\wedge^*(q^{-1}\mf{g}_{\mbf{v}}[-1])
\end{equation}

Thus we have that
\begin{align}
\lim_{a\rightarrow0}\frac{\wedge^*(q^{-1}\mf{g}_{\mbf{v}}[-1])}{\wedge^*(q^{-1}\mf{g}_{\mbf{v}}[-1])}=1
\end{align}

Thus we can cancel the term contributed by $\text{Stab}_{\infty}^{-1}$, $\text{Stab}_{\infty,U}$ and $1/\wedge^*(q\mf{g}_{\mbf{v}}[-1])$. Also note that $\text{Stab}_{U_1}(\alpha)\otimes\text{Stab}_{U_2}(\beta)$ descends to the class in $Y_{\mbf{v}_1,\mbf{v}_1+\mbf{w}_1}/G_{\mbf{v}_1}\times Y_{\mbf{v}_2,\mbf{v}_2+\mbf{w}_2}/G_{\mbf{v}_2}=M(\mbf{v}_1,\mbf{v}_1+\mbf{w}_1)_{iso}/G_{\mbf{v}_1}'\times M(\mbf{v}_2,\mbf{v}_2+\mbf{w}_2)_{iso}/G_{\mbf{v}_2}'$. It can be illustrated by the following diagram:
\begin{equation}
\begin{tikzcd}
K_T(M(\mbf{v},\mbf{v}+\mbf{w}))\arrow[d,"i_U^*"]\arrow[r,"j^*"]&K_{T}(\mc{R}_{\mbf{v}})\arrow[d,"i^*"]\\
K_{T}(M(\mbf{v}_1,\mbf{w}_1+\mbf{v}_1)\times M(\mbf{v}_2,\mbf{w}_2+\mbf{v}_2))\arrow[r,"(j_1\times j_2)^*"]&K_{T}(\mc{R}_{\mbf{v}_1}\times\mc{R}_{\mbf{v}_2})
\end{tikzcd}
\end{equation}
So we have that 
\begin{equation}
\begin{aligned}
&\lim_{a\rightarrow0}\frac{1}{\wedge^*(N_{X^A}^-)}i^*\frac{1}{\wedge^*(q\mf{g}_{\mbf{v}}[1])\wedge^*(q\mf{g}_{\mbf{v}}[-1])}\frac{1}{\wedge^*(q\mf{g}_{\mbf{v}_1})\wedge^*(q\mf{g}_{\mbf{v}_2})}j^*(i^*_U)^{-1}(\wedge^*(N_{X'^A}^-)\text{Stab}_{U_1}(\alpha)\otimes\text{Stab}_{U_2}(\beta))\\
=&\lim_{a\rightarrow0}\frac{\mc{L}^{-1}}{\wedge^*(q\mf{g}_{\mbf{v}_1})\wedge^*(q\mf{g}_{\mbf{v}_2})}i^*j^*(i^*_U)^{-1}(\mc{L}\text{Stab}_{U_1}(\alpha)\otimes\text{Stab}_{U_2}(\beta))\\
=&\frac{1}{\wedge^*(q\mf{g}_{\mbf{v}_1})\wedge^*(q\mf{g}_{\mbf{v}_2})}(j_1\times j_2)^*(\text{Stab}_{U_1}(\alpha)\otimes\text{Stab}_{U_2}(\beta))=g_{\alpha}\otimes g_{\beta}
\end{aligned}
\end{equation}

So we have that as $a\rightarrow0$, the diagonal part of $i^*(g(\text{Stab}_{s}(\alpha\otimes\beta)))$ can be written as:
\begin{align}
\lim_{a\rightarrow0}a^{(\cdots)}i^*(g(\text{Stab}_{\infty}(\alpha\otimes\beta)))=\lim_{a\rightarrow0}a^{(\cdots)}\wedge^*(N_{X^A}^-)(g_{\alpha}\otimes g_{\beta})\in K_{T}(\mc{R}_{\mbf{v}_1}\times\mc{R}_{\mbf{v}_2})
\end{align}

while the degree for $\wedge^*(N_{X^A}^-)$ coincides with the degree of $\text{Stab}_{\infty}$, and it will become a line bundle $\mc{L}$ under the limit, and so we have furthermore reduction of \ref{reduction-of-adjoint-capping} that:
\begin{equation}
\begin{aligned}
&\lim_{a\rightarrow0}\text{Stab}_{\infty}^{-1}\Psi^{\dagger}_{X}\text{Stab}_{s}(\alpha\otimes\beta)=\Psi^{\dagger}_{X^A}(\alpha\otimes\beta)(zq^{\frac{\Delta\mbf{w}_2-C\mbf{v}_2}{2}}, zq^{\frac{C\mbf{v}_1-\Delta\mbf{w}_1}{2}})\text{ When restricted to the diagonal}
\end{aligned}
\end{equation}

So now we define the operator:
\begin{align}
\mbf{J}^{(\mbf{w}_1),(\mbf{w}_2),\dagger}=(q^{\Omega}\Psi^{\dagger}_{X^A}(zq^{\frac{\Delta\mbf{w}_2-C\mbf{v}_2}{2}}, zq^{\frac{C\mbf{v}_1-\Delta\mbf{w}_1}{2}})^{-1}i^*\Psi_{X}^{\dagger}(z)(i^*)^{-1}
\end{align}

Furthermore we define that:
\begin{align}
H^{(\mbf{w}_1),(\mbf{w}_2),\dagger}:=\lim_{a\rightarrow0}\text{Stab}_{\infty}^{-1}\mbf{J}^{(\mbf{w}_1),(\mbf{w}_2),\dagger}\text{Stab}_{s}
\end{align}

Now we use the following identity:
\begin{align}
\Psi_{X}(z)\text{Vertex}^{(g_{\alpha})}_{X}(z)=\mbf{G}_{X}(\alpha)(z)
\end{align}

Conjugate the identity by $\text{Stab}_{s}$ and we obtain that:
\begin{equation}
\begin{aligned}
\text{Stab}_s^{-1}\Psi_{X}(z)\text{Vertex}^{(g_{\text{Stab}_{s}(\alpha\otimes\beta)})}_{X}(z)=\text{Stab}_{s}^{-1}\mbf{G}_{X}(z)\text{Stab}_s(\alpha\otimes\beta)
\end{aligned}
\end{equation}

Now we prove the following result:
\begin{lem}
For the gluing operator $\text{Stab}_{s}^{-1}\mbf{G}_{X}(z)\text{Stab}_{s}:K_{T}(X^A)\rightarrow K_{T}(X^A)[[z]]$ conjugate by the stable envelopes, it is a block diagonal operator and the block diagonal component is given by $\mbf{G}_{X^A}(zq^{(\cdots)},zq^{(\cdots)})$. 
\end{lem}
\begin{proof}
There are two ways to prove this fact. By the Theorem 10.3.2 in \cite{O15} using the virtual localisation, the operator $\text{Stab}_s^{-1}\mbf{G}\text{Stab}_s$ is independent of the equivariant parametre $a$. Thus there is no contribution from the broken nodes, and thus using the same argument as in the proof of one can prove that when restricted to the block diagonal part, there is no contribution from the broken nodes. Thus using the argument of the proposition \ref{digaonal-argument}, we have that:
\begin{align}
\Delta_{s}(\mbf{G}_{X}(z))=q^{\Omega}\mbf{G}_{F_0}(zq^{\frac{\Delta\mbf{w}_2-C\mbf{v}_2}{2}}, zq^{\frac{C\mbf{v}_1-\Delta\mbf{w}_1}{2}})
\end{align}

The proof can also be carried out in the following way: Recall that the gluing operator $\mbf{G}$ is the limit of the quantum difference operator:
\begin{align}
\mbf{G}(z)=\lim_{p\rightarrow\infty}\mbf{M}_{\mc{L}}(zp^{-\mc{L}},p)\mc{L}^{-1}=\lim_{p\rightarrow\infty}\text{Const}\cdot\mc{L}\prod_{w}^{\leftarrow}\mbf{B}_{w}(zp^{-\mc{L}})\mc{L}^{-1}
\end{align}

Now the conjugation by $\text{Stab}_{s}$ is equivalent to the coproduct $\Delta_{s}$ on $\mbf{G}_{X}$. It has been proved in Proposition 15 in \cite{OS22} that if we restrict $\Delta_{s}(\mbf{G}_{X})$ to the block diagonal operator $K_{T}(F_{\alpha})\rightarrow K_{T}(F_{\alpha})$. 

\end{proof}

Now if we take the limit $a\rightarrow0$ we have that:
\begin{equation}
\begin{aligned}
H^{(\mbf{w}_1),(\mbf{w}_2)}(z)q^{\Omega}\mbf{G}_{X^A}((zq^{\frac{\Delta\mbf{w}_2-C\mbf{v}_2}{2}}, zq^{\frac{C\mbf{v}_1-\Delta\mbf{w}_1}{2}})q^{\Omega}H^{(\mbf{w}_1),(\mbf{w}_2),\dagger}(z)=q^{\Omega}\mbf{G}_{X^A}(zq^{\frac{\Delta\mbf{w}_2-C\mbf{v}_2}{2}}, zq^{\frac{C\mbf{v}_1-\Delta\mbf{w}_1}{2}})
\end{aligned}
\end{equation}

Since $G_{X^A}(z)$ is a block-diagonal operator, we obtain that:
\begin{align}
H^{(\mbf{w}_1),(\mbf{w}_2),\dagger}(z)=\mbf{G}_{X^A}(zq^{\frac{\Delta\mbf{w}_2-C\mbf{v}_2}{2}}, zq^{\frac{C\mbf{v}_1-\Delta\mbf{w}_1}{2}})^{-1}q^{-\Omega}(H^{(\mbf{w}_1),(\mbf{w}_2)}(z))^{-1}q^{\Omega}\mbf{G}_{X^A}(zq^{\frac{\Delta\mbf{w}_2-C\mbf{v}_2}{2}}, zq^{\frac{C\mbf{v}_1-\Delta\mbf{w}_1}{2}})
\end{align}

Thus each block part $(H^{\dagger}_{\alpha\beta})$ of $H^{(\mbf{w}_1),(\mbf{w}_2),\dagger}$ is given by:
\begin{align}
q^{\Omega}H^{\dagger}_{\alpha\beta}(z)=G_{F_{\alpha}}^{-1}q^{-\Omega}H^{-1}_{\alpha\beta}q^{\Omega}G_{F_{\beta}}
\end{align}

Thus we can see that $q^{\Omega}H^{(\mbf{w}_1),(\mbf{w}_2),\dagger}(z)$ is also a strictly block-triangular operator, and this operator can be thought of the adjoint of the operator $H^{(\mbf{w}_1),(\mbf{w}_2)}(z)$.

\section{\textbf{Rationality of the capped vertex function}}\label{section-6}

In this section we prove the main theorem of the paper:
\begin{thm}\label{Main-theorem}
The capped vertex function $\text{Cap}^{(\tau)}(z)\in K_{T}(X)_{loc}[[z^{\mbf{d}}]]$ is a rational function in $K_{T}(X)_{loc}(z^{\mbf{d}})$ for arbitrary $\tau\in K_{T}(pt)[\cdots,x_{i1}^{\pm1},\cdots,x_{in_i}^{\pm1},\cdots]^{\text{Sym}}$.
\end{thm}

\subsection{The main theorem for the special case}
We first focus on the case when 
\begin{equation}
\tau\in K_{T}(pt)[\cdots,x_{i1},\cdots,x_{in_i},\cdots]^{\text{Sym}}
\end{equation}
is a symmetric polynomial rather than a Laurent polynomial.

\begin{prop}\label{Main-prop-1}
The capped vertex function $\text{Cap}^{(\tau)}(z)\in K_{T}(X)_{loc}[[z^{\mbf{d}}]]$ is a rational function in $K_{T}(X)_{loc}(z^{\mbf{d}})$ for arbitrary $\tau\in K_{T}(pt)[\cdots,x_{i1},\cdots,x_{in_i},\cdots]^{\text{Sym}}$.
\end{prop}

We choose the polarisation $T^{1/2}$ and $\mbf{w}$ as in \ref{polarisaion-large-framing} such that it satisfies the large framing vanishing condition for some $\bm{\theta}_*\text{deg }\tau$. Since such polarisation is only the shift of the polarisation \ref{polarisation-0} by a large line bundle $\mc{L}-q^{-1}\mc{L}$, and the line bundle only affects the capped vertex with the multiplication by some $p^{(\cdots)}$. The shift by some $p^{(\cdots)}$ will not affect the proof for the rationality of the capped vertex function. So here for simplicity we ignore this shift by $\mc{L}-q^{-1}\mc{L}^{-1}$.

Now we fix the stability condition $\bm{\theta}$ and $\tau(\mc{V})$:
\begin{align}
\text{Cap}^{(\tau)}(z,a)=\Psi(z,a)\text{Vertex}^{(\tau)}(z,a)
\end{align}

Now choose the torus $A$ such that $\mbf{w}=\mbf{w}_1+a\mbf{w}_2$, and we choose the embedding $M(\mbf{v}_1,\mbf{w}_1)\times M(\mbf{v}_2,\mbf{w}_2)\hookrightarrow M(\mbf{v}_1+\mbf{v}_2,\mbf{w}_1+\mbf{w}_2)$.
In the case when $\mbf{w}$ is really large,for the case that $\tau$ is a symmetric polynomial of tautological classes. We have:
\begin{align}
\tau(\mc{V})\mc{K}^{1/2}=\Psi(z,a)\text{Vertex}^{(\tau)}(z,a)
\end{align}

Now we take the splitting $\mbf{w}=\mbf{w}_1+a\mbf{w}_2$, and now as $a\rightarrow0$, on the left hand side it would degenerate to $(\tau(\mc{V})\otimes1)\mc{K}^{1/2}$. So we have:
\begin{align}
(\tau(\mc{V})\otimes1)\mc{K}^{1/2}=(\Psi(z,0)\text{Vertex}^{(\tau)}_X(z,0))
\end{align}
and on the right hand side, it would degenerate to:
\begin{equation}
\begin{aligned}
&\lim_{a\rightarrow0}i^*\text{Stab}_{s}H^{(\mbf{w}_1),(\mbf{w}_2)}(z)(i^*\text{Stab}_{\infty})^{-1}\Psi^{(\mbf{w}_1,\mbf{w}_2)}_{X^A}(zq^{\frac{\Delta\mbf{w}_2-C\mbf{v}_2}{2}}, zq^{\frac{C\mbf{v}_1-\Delta\mbf{w}_1}{2}})\text{Vertex}^{(i^*\tau)}_{X^A}(zq^{\frac{\Delta\mbf{w}_2-C\mbf{v}_2}{2}}, zq^{\frac{C\mbf{v}_1-\Delta\mbf{w}_1}{2}}p^{C\mbf{v}_1-\Delta\mbf{w}_1})\\
=&q^{\Omega}H^{(\mbf{w}_1),(\mbf{w}_2)}(z)\Psi^{(\mbf{w}_1,\mbf{w}_2)}_{X^A}(zq^{\frac{\Delta\mbf{w}_2-C\mbf{v}_2}{2}}, zq^{\frac{C\mbf{v}_1-\Delta\mbf{w}_1}{2}})\text{Vertex}^{(i^*\tau)}_{X^A}(zq^{\frac{\Delta\mbf{w}_2-C\mbf{v}_2}{2}}, zq^{\frac{C\mbf{v}_1-\Delta\mbf{w}_1}{2}}p^{C\mbf{v}_1-\Delta\mbf{w}_1})
\end{aligned}
\end{equation}

Here $X^A=\bigsqcup_{\mbf{v}_1+\mbf{v}_2=\mbf{v}}M(\mbf{v}_1,\mbf{w}_1)\otimes M(\mbf{v}_2,\mbf{w}_2)$. Now we choose the component $M(\mbf{v},\mbf{w}_1)\times M(0,\mbf{w}_2)$. In this case the right hand side vertex function $\text{Vertex}^{(\tau)}_{X^A}(z)$ can be reduced to $\text{Vertex}_{M(\mbf{v},\mbf{w}_1)}^{(\tau)}(z)\otimes1$. Thus now we have that:

\begin{align}
H^{(\mbf{w}_1),(\mbf{w}_2)}(z)^{-1}\mc{L}^{-1}(\tau(\mc{V})\otimes1)\mc{K}^{1/2})=q^{\Omega}\mc{L}^{-1}(\Psi^{(\mbf{w}_1)}(zq^{\frac{\Delta\mbf{w}_2}{2}})\otimes1))(\text{Vertex}^{(\tau)}_{M(\mbf{v},\mbf{w}_1)}(zq^{\frac{\Delta\mbf{w}_2}{2}})\otimes1)
\end{align}

Thus we have that:
\begin{align}
\text{Cap}^{(\tau)}(zq^{\frac{\Delta\mbf{w}_2}{2}})\otimes1=\mc{L}q^{-\Omega}H^{(\mbf{w}_1),(\mbf{w}_2)}(z)^{-1}\mc{L}^{-1}(\tau(\mc{V})\otimes1)\mc{K}^{1/2})
\end{align}

Here the term $1$ on the left hand side is just the trivial one, while the $1$ on the right hand side corresponds to the structure sheaf $\mc{O}$ for the quiver varieties $M(\mbf{v}_2,\mbf{w}_2)$ for which the $\mbf{v}_2$ is not zero.

Now since $H^{(\mbf{w}_1),(\mbf{w}_2)}(z)$ has rational coefficients as $z$, this implies that $\text{Cap}^{(\tau)}(z,u)$ is rational. Thus the proof is finished for $\tau\in K_{T}(pt)[\cdots,x_{i1},\cdots,x_{in_i},\cdots]^{\text{Sym}}$.

Before we proceed to the proof of the general descendents $\tau\in K_{T}(pt)[\cdots,x_{i1}^{\pm1},\cdots,x_{in_i}^{\pm1},\cdots]^{\text{Sym}}$, we need to prove the rationality of the quantum difference operators, which can be thought of as an application of the result in this subsection.

\subsection{\textbf{Rationality of the quantum difference operators}}
The rationality of the capped vertex operator with descendents also has an interesting application: It can be used to prove the rationality of the quantum difference operators.

First let us define $\mc{A}_{\mc{L}}:=T_{\mc{L}}^{-1}\mbf{M}_{\mc{L}}$ the corresponding difference operator. It is well-known that:
\begin{align}
\mc{A}_{\mc{L}}\mc{A}_{\mc{L}'}=\mc{A}_{\mc{L}\mc{L}'}=\mc{A}_{\mc{L}'}\mc{A}_{\mc{L}}
\end{align}

It has been proved in \cite{OS22} that quantum difference operator $\mc{A}_{\mc{L}}$ can be expressed by the operators in the Maulik-Okounkov quantum affine algebra:
\begin{align}\label{qde-transform}
\text{Stab}_{s,+}^{-1}\mc{A}_{\mc{L}}\text{Stab}_{s,+}=\text{Const}^s_{\mc{L}}\mc{A}^s_{\mc{L}}
\end{align}

Here $\mc{A}^s_{\mc{L}}=T_{\mc{L}}^{-1}\mc{B}^s_{\mc{L}}$, $\mc{B}^s_{\mc{L}}=\mc{L}\prod_{w}\mbf{B}_{w}\in U_{q}^{MO}(\hat{\mf{g}}_{Q})$ is the product of the monodromy operators defined in \cite{OS22}. By definition, $\mc{A}_{\mc{L}}^s$ has the matrix coefficients as the rational functions in $u,z,t_e$. It is also easy to check that $\text{Const}^s_{\mc{L}}\text{Const}^s_{\mc{L}'}=\text{Const}^s_{\mc{L}\mc{L}'}$.

Using the rationality of the capped vertex function, we can prove the rationality of the quantum difference operators:
\begin{prop}\label{rationality-of-qde}
The quantum difference operator $\mbf{M}_{\mc{L}}(z)$ has the matrix coefficients as the rational function of $z,u,t_e$.
\end{prop}
\begin{proof}
First we choose a quiver variety $M(\mbf{v},\mbf{w})$ such that $\mbf{w}$ is large enough to satisfy the large framing vanishing condition. Consider its corresponding capped vertex function with descendents $\tau$:
\begin{align}
\text{Cap}^{(\tau)}(z,a)=\Psi(z,a)\text{Vertex}^{(\tau)}(z,a)
\end{align}
equivalently:
\begin{align}
\tau(\mc{V})\mc{K}^{1/2}=\Psi(z,a)\text{Vertex}^{(\tau)}(z,a)
\end{align}
Now choose a torus action $A\subset T$ such that $\mbf{w}=\mbf{w}_1+a\mbf{w}_2$. Thus now we have that:
\begin{equation}
\begin{aligned}
&(\tau(\mc{V})\otimes1)\mc{K}^{1/2}\\
=&Y^{(\mbf{w}_1),(\mbf{w}_2)}(z)(\Psi^{(\mbf{w}_1)}(zq^{\frac{\Delta\mbf{w}_2-C\mbf{v}_2}{2}})\otimes\Psi^{(\mbf{w}_2)}(zq^{\frac{C\mbf{v}_1-\Delta\mbf{w}_1}{2}}))(\text{Vertex}^{(\tau)}(zq^{-\frac{C\mbf{v}_2-\Delta\mbf{w}_2}{2}}\otimes \text{Vertex}^{(1)}(zq^{\frac{C\mbf{v}_1-\Delta\mbf{w}_1}{2}}p^{C\mbf{v}_1-\Delta\mbf{w}_1})))\\
=&Y^{(\mbf{w}_1),(\mbf{w}_2)}(z)(\text{Cap}^{(\tau)}(zq^{\frac{\Delta\mbf{w}_2-C\mbf{v}_2}{2}})\otimes\mbf{M}_{C\mbf{v}_1-\Delta\mbf{w}_1}(z)^{-1}\text{Cap}^{(1)}(zq^{\frac{C\mbf{v}_1-\Delta\mbf{w}_1}{2}}p^{C\mbf{v}_1-\Delta\mbf{w}_1})\mc{L}^{C\mbf{v}_1-\Delta\mbf{w}_1})
\end{aligned}
\end{equation}
Here $\mbf{M}_{\mbf{n}}(z)$ with the integer vector $\mbf{n}=(n_1,\cdots,n_k)\in\mbb{Z}^k$ corresponds to the quantum difference operator $\mbf{M}_{\mc{L}}(z)$ such that $\mc{L}=\mc{L}_1^{n_1}\cdots\mc{L}_{k}^{n_k}$.

The rationality of capped vertex function and the operator $Y^{(\mbf{w}_1),(\mbf{w}_2)}(u,z)$ implies that the matrix coefficients of $\mbf{M}_{C\mbf{v}_1-\Delta\mbf{w}_1}(z)\in\text{End}(K_{T}(M(\mbf{v}_2,\mbf{w}_2)))[[z^{\mbf{d}}]]$ is rational in $z$. If we choose $\mbf{w}$ to satisfy the large framing condition ($w_i=C|\theta_i|+B_i$ with $C$ really large) and fix $\mbf{v}_2$ and $\mbf{w}_2$, the choice of $\mbf{v}_1$ and $\mbf{w}_1$ is arbitrary, this implies that the quantum difference operator $\mbf{M}_{\mc{L}}(z)$ is rational for some component $K_{T}(M(\mbf{v}_2,\mbf{w}_2))$ with fixed $\mbf{w}_2$ and arbitrary $\mc{L}$.

Then using the formula $\ref{qde-transform}$, we can see that the constant operator $\text{Const}_{\mc{L}}$ is rational in $z$. But the constant operator $\text{Const}_{\mc{L}}$ is independent of the choice of the component $K_{T}(M(\mbf{v},\mbf{w}))$, combining with the fact that $\text{Stab}_{+,s}$ and $\mc{B}^s_{\mc{L}}$ are rational, we conclude that $\mbf{M}_{\mc{L}}(z)$ is rational in $z$.
\end{proof}

\subsection{The main theorem for the general case}

For the general case 
\begin{align}
\tau(\mc{V})\in K_{T}(pt)[\cdots,x_{i1}^{\pm1},\cdots,x_{in_i}^{\pm1},\cdots]^{\text{Sym}}
\end{align}
we choose an ample line bundle $\mc{L}\in K_{T}(pt)[\cdots,x_{i1}^{\pm1},\cdots,x_{in_i}^{\pm1},\cdots]^{\text{Sym}}$ such that
\begin{align}
\tau(\mc{V})\otimes\mc{L}\in K_{T}(pt)[\cdots,x_{i1},\cdots,x_{in_i},\cdots]^{\text{Sym}}
\end{align}

Now we have that:

\begin{equation}
\begin{aligned}
&\text{Cap}^{(\tau\otimes\mc{L})}(a,z)=\Psi(a,z)\text{Vertex}^{(\tau\otimes\mc{L})}(a,z)=\Psi(a,z)\mc{L}\text{Vertex}^{(\tau)}(a,zp^{\mc{L}})
\end{aligned}
\end{equation}

Here we use the translation symmetry of the vertex function with descendents:
\begin{align}
\text{Vertex}^{(\tau\otimes\mc{L})}(a,z)=\mc{L}\text{Vertex}^{(\tau)}(a,zp^{\mc{L}})
\end{align}

Then combining with the capping operator $\text{Cap}^{(\tau)}(a,z)$, we have that:

\begin{equation}
\begin{aligned}
&\text{Cap}^{(\tau)}(a,zp^{\mc{L}})=\Psi(a,zp^{\mc{L}})\text{Vertex}^{(\tau)}(a,zp^{\mc{L}})=\mbf{M}_{\mc{L}}(z)\Psi(a,z)\mc{L}^{-2}\text{Vertex}^{(\tau\otimes\mc{L})}(a,z)\\
=&\mbf{M}_{\mc{L}}(z)\Psi(a,z)\text{Vertex}^{(\tau\otimes\mc{L}^{-1})}(a,zp^{2\mc{L}})=\mbf{M}_{\mc{L}}(z)\text{Cap}^{(\tau\otimes\mc{L}^{-1})}(a,zp^{2\mc{L}})
\end{aligned}
\end{equation}

Thus we have the following formula:
\begin{align}\label{general-cap}
\text{Cap}^{(\tau)}(a,zp^{\mc{L}})=\mbf{M}_{\mc{L}}(zp^{-\mc{L}})^{-1}\text{Cap}^{(\tau\otimes\mc{L})}(a,z)
\end{align}

Since $\mbf{M}_{\mc{L}}(z)$ is rational in $z$ by \ref{rationality-of-qde}, we can see that $\text{Cap}^{(\tau)}(a,z)$ is rational in $z$. This finish the proof of the main theorem.

\textbf{Remark.}
One application of the result is the eigenvector of the quantum difference operator $\mbf{M}_{\mc{L}}(z)$ in the large framing condition. In this case the capping operator $\text{Cap}^{(\tau\otimes\mc{L})}(a,z)$ and $\text{Cap}^{(\tau\otimes\mc{L})}(a,z)$ would descend to $(\tau\otimes\mc{L})K_{X}^{1/2}$ and $\tau\otimes K_{X}^{1/2}$. Therefore we have that:
\begin{align}
\mbf{M}_{\mc{L}}(z)\tau(\mc{V})K_{X}^{1/2}=\mc{L}\otimes\tau(\mc{V})K_{X}^{1/2}
\end{align}

We can see that in the large framing limit, the tautological class $\tau(\mc{V})K_{X}^{1/2}$ is an eigenvector of $\mbf{M}_{\mc{L}}(z)$ with the eigenvalue given by $\mc{L}$.

\subsection{GIT wall-crossing for the capped vertex function}

We define the matrix coefficients of the capped vertex function as:
\begin{align}
\chi^{(\tau)}_{\bm{\theta}}(z)=\chi(\text{Cap}^{(\tau)}_{\bm{\theta}})(z)\in K_{T}(pt)[[z]]
\end{align}
By the theorem above we know that this is a rational function in both equivariant variable and the Kahler variable $z$.

The goal of this subsection is to compute the GIT wall-crossing for the function $\chi^{(\tau)}_{\bm{\theta}}(z)$.

We recall a general integral formula from the appendix of \cite{AFO18} for the Euler characteristic for the locally free sheaf $\mc{F}$ over $X$:
\begin{align}
\chi(X_{\bm{\theta}},\mc{F})=\frac{1}{|W|}\int_{\chi_{\bm{\theta}}}\Delta_{\text{Weyl}}(s)\frac{\mc{F}(s)}{\wedge^*(T^*\text{Rep}(\mbf{v},\mbf{w}))}d_{Haar}s
\end{align}

Here the integral contour $\chi_{\bm{\theta}}$ is chosen as the image of the contour $\{|s|=1\}\subset A_{\mbf{v}}^{o}:=\{s|\forall w_{k,i}(s)\neq1\}\subset A_{\mbf{v}}\subset G_{\mbf{v}}$ in the middle homology $H_{\text{mid}}(A_{\mbf{v}}^{o},\{|\bm{\theta}|<<1\},\mbb{C})$.

Now for two different generic stability conditions $\bm{\theta}$ and $\bm{\theta}'$ and the corresponding quiver varieties $M_{\bm{\theta}}(\mbf{v},\mbf{w})$ and $M_{\bm{\theta}'}(\mbf{v},\mbf{w})$, we choose $\mbf{w}$ which is large for both $\bm{\theta}$ and $\bm{\theta}'$. Note that this requires that $\bm{\theta}$ and $\bm{\theta}'$ are not that "far away", i.e. each component $\theta_i,\theta_i'$ of $\bm{\theta}$ and $\bm{\theta}'$ are closed. In this case the capped operator $\text{Cap}^{(\tau)}$ has the form $\tau\mc{K}^{1/2}$ for both stability conditions. i.e. They have the same preimage in $K_{T}(M_{0}(\mbf{v},\mbf{w}))$ via the Kirwan map:
\begin{equation}
\begin{tikzcd}
K_{T}(M_{\bm{\theta}'}(\mbf{v},\mbf{w}))&K_{T}(M_{0}(\mbf{v},\mbf{w}))\arrow[r,"i_{\bm{\theta}}^*"]\arrow[l,"i_{\bm{\theta}'}^*"]&K_{T}(M_{\bm{\theta}}(\mbf{v},\mbf{w}))
\end{tikzcd}
\end{equation}

For the integral formula above, we can think about $\mc{F}$ as an element in $K_{T}(M_{0}(\mbf{v},\mbf{w}))$ generated by the tautological classes. Thus if we fix $\mc{F}\in K_{T}(M_{0}(\mbf{v},\mbf{w}))$, the continuous change of the stability condition in $\chi(X_{\bm{\theta}},\mc{F})$ can be thought of as the analytic contination of the integral in the variable of the equivariant variable over $T$ as long as the integral converges.

For the capped vertex with descendents $\tau\in K_{T}(pt)[\cdots,x_{i1},\cdots,x_{in_i},\cdots]^{\text{Sym}}$, if we fix the stability condition $\bm{\theta}$, recall the formula:
\begin{align}
\text{Cap}^{(\tau)}_{\bm{\theta}}(zq^{\frac{\Delta\mbf{w}_2}{2}})\otimes1=\mc{L}_{\bm{\theta}}q^{-\Omega}H^{(\mbf{w}_1),(\mbf{w}_2)}_{\bm{\theta}}(z)^{-1}\mc{L}^{-1}_{\bm{\theta}}(\tau(\mc{V})\otimes1)\mc{K}^{1/2})
\end{align}

Thus the corresponding integral is given as:
\begin{align}
\chi^{(\tau)}_{\bm{\theta}}(zq^{\frac{\Delta\mbf{w}_2}{2}})=\frac{1}{|W|}\int_{\chi_{\bm{\theta}}}\Delta_{\text{Weyl}}(s)\mc{L}_{\bm{\theta}}q^{-\Omega}H^{(\mbf{w}_1),(\mbf{w}_2)}_{\bm{\theta}}(z)^{-1}\mc{L}^{-1}_{\bm{\theta}}(\tau(\mc{V})\otimes1)\mc{K}^{1/2})d_{Haar}s
\end{align}

For the general $\tau\in K_{T}(pt)[\cdots,x_{i1}^{\pm1},\cdots,x_{in_i}^{\pm1},\cdots]^{\text{Sym}}$. Using the formula in \ref{general-cap}, the GIT wall-crossing can be expressed as follows:
\begin{thm}\label{wall-crossing formula}
For two stability condition $\bm{\theta}$ and $\bm{\theta}'$, the GIT wall-crossing for the Euler characteristic of the capped vertex function with descendents $\tau\in K_{T}(pt)[\cdots,x_{i1}^{\pm1},\cdots,x_{in_i}^{\pm1},\cdots]^{\text{Sym}}$ is given by:
\begin{equation}
\begin{aligned}
&\chi^{(\tau)}_{\bm{\theta}}(zp^{\mc{L}})-\chi^{(\tau)}_{\bm{\theta}'}(zp^{\mc{L}})\\
=&H_{\bm{\theta},\bm{\theta}'}^{(\tau)}+\frac{1}{|W|}\int_{\chi_{\bm{\theta}'}}\frac{\Delta_{\text{Weyl}}(s)}{\wedge^*(T^*\text{Rep}(\mbf{v},\mbf{w}))}q^{-\Omega}\times\\
&\times(\mbf{M}_{\mc{L}}^{\bm{\theta}}(zp^{-\mc{L}})^{-1}\mc{L}_{\bm{\theta}}H^{(\mbf{w}_1),(\mbf{w}_2)}_{\bm{\theta}}(zq^{-\frac{\Delta\mbf{w}_2}{2}})^{-1}\mc{L}^{-1}_{\bm{\theta}}-\mbf{M}_{\mc{L}}^{\bm{\theta}'}(zp^{-\mc{L}})^{-1}\mc{L}_{\bm{\theta}'}H^{(\mbf{w}_1),(\mbf{w}_2)}_{\bm{\theta}'}(zq^{-\frac{\Delta\mbf{w}_2}{2}})^{-1}\mc{L}^{-1}_{\bm{\theta}'})\times\\
&\times((\tau(\mc{V})\otimes\mc{L})\otimes1)\mc{K}^{1/2})d_{Haar}s
\end{aligned}
\end{equation}

Here $H^{(\tau)}_{\bm{\theta},\bm{\theta}'}(z)$ is the analytic continuation of $\chi^{(\tau)}_{\bm{\theta}}(z)$ via contour deformation from $\chi_{\bm{\theta}}$ to $\chi_{\bm{\theta}'}$. $\mc{L}$ is some ample line bundle such that it makes $\tau\otimes\mc{L}\in K_{T}(pt)[\cdots,x_{i1},\cdots,x_{in_i},\cdots]^{\text{Sym}}$.
\end{thm}

For $\bm{\theta}$ and $\bm{\theta}'$ such that the corresponding contour is really close, $H^{(\tau)}_{\bm{\theta},\bm{\theta}'}(z)=0$. Thus in this case the GIT wall-crossing only depends on the difference of the operator $H^{(\mbf{w}_1),(\mbf{w}_2)}_{\bm{\theta}}(z)^{-1}$ of different stability condition $\bm{\theta}$.

For the choice of the stability condition $\bm{\theta}=(1,\cdots,1)$ and $\bm{\theta}'=-\bm{\theta}$, usually the corresponding Maulik-Okounkov quantum affine algebra $U_{q}(\hat{\mf{g}}_{Q})$ is isomorphic to the double of the corresponding preprojective $K$-theoretic Hall algebra $\mc{A}_{Q}$. The corresponding action on $K_{T}(M_{\bm{\theta}}(\mbf{w}))$ for $\bm{\theta}$ and $-\bm{\theta}$ is dual to each other due to the construction given in \cite{YZ18}.  In this case it is expected that the fusion operator $H^{(\mbf{w}_1),(\mbf{w}_2)}(z)$ and the quantum difference operator $\mbf{M}_{\mc{L}}(z)$ can be solved in the slope subalgebra $\mc{A}_{Q}$.

In the case when the quiver $Q$ is finite, Jordan or affine type $A$, it has been proved that the MO quantum affine algebra of $\bm{\theta}=(1,\cdots,1)$ is isomorphic to $\mc{A}_{Q}$ in \cite{Z24} and \cite{Z24-1}. The corresponding capped vertex function and the fusion operator has been computed in \cite{AS24} and \cite{AD24}. Thus in these cases the GIT wall-crossing can be computed explicitly.

\end{document}